\documentclass{elsarticle}

\usepackage{amsmath,amssymb,amsthm}
\usepackage{bbm}
\usepackage[all]{xy}

\journal{Advances in Mathematics}

\newtheorem{Thm}{Theorem}
\newtheorem{Cor}[Thm]{Corollary}
\newtheorem{theorem}{Theorem}[section]
\newtheorem{lemma}[theorem]{Lemma}
\newtheorem{corollary}[theorem]{Corollary}
\newtheorem{proposition}[theorem]{Proposition}

\theoremstyle{definition}
\newtheorem{remark}[theorem]{Remark}
\newtheorem{definition}[theorem]{Definition}
\newtheorem{examples}[theorem]{Examples}

\def\RP{\mathbbm{R P}}
\def\R{\mathbbm{R}}
\def\F{\mathbbm{F}}

\def\P{\mathbbm{P}}
\def\D{\mathbbm{D}}
\def\C{\mathbbm{C}}
\def\s{\mathbbm{S}}
\def\B{\mathbbm{B}}
\def\X{\mathcal X}
\def\cL{\mathcal L}
\def\a{\alpha}
\def\e{\varepsilon}
\def\t{\theta}
\def\0{\underline 0}
\newcommand{\inpr}[2]{\ensuremath{\langle#1,#2\rangle}}
\newcommand{\abs}[1]{\lvert #1\rvert}   
\newcommand{\norm}[1]{\lVert #1\rVert}  
\newcommand{\set}[2]{\ensuremath{\{\,{#1}\mid {#2}\,\}}}
\DeclareMathOperator{\grad}{grad}
\DeclareMathOperator{\dist}{dist}
\DeclareMathOperator{\sing}{Sing}

\DeclareMathOperator{\spn}{span}

\begin{document}

\begin{frontmatter}



\title{Refinements of Milnor's Fibration Theorem for Complex
Singularities\tnoteref{support}}


\author{J.~L.~Cisneros-Molina\fnref{ictp}}
\ead{jlcm@matcuer.unam.mx}

\author{J.~Seade\corref{cor1}}
\ead{jseade@matcuer.unam.mx}

\author{J.~Snoussi}
\ead{jsnoussi@matcuer.unam.mx}
\fntext[ictp]{Regular Associate of the 
International Centre for Theoretical Physics, Trieste, Italy.}

\tnotetext[support]{Partially supported by CONACYT: G-36357-E, J-49048-F,
U-55084, Mexico, DGAPA-UNAM: PAPIIT IN105806, IN102208, Mexico,  by ICTP,
Italy, and CNRS, France.}

\cortext[cor1]{Corresponding author}

\address{Instituto de Matem\'aticas, Unidad Cuernavaca,\\
Universidad Nacional Aut\'onoma de M\'exico,\\
Av. Universidad s/n, Lomas de Chamilpa,
C.P.~62210, Cuernavaca, Morelos, M\'exico}

\begin{abstract}
Let $X$ be an analytic subset of an open neighbourhood  $U$ of the
origin $\0$ in $\C^n$. Let $f\colon (X,\0) \to (\C,0)$ be
holomorphic  and set $V =f^{-1}(0)$. Let $\B_\e$ be a ball in $U$
of sufficiently small radius $\e>0$, centred at $\0\in\C^n$. We
show that $f$ has an associated canonical pencil of real analytic
hypersurfaces $X_\t$, with axis $V$,  which leads to a fibration
$\Phi$ of the whole space $(X \cap \mathbb{B}_\e) \setminus V$
over $\s^1 $. Its restriction  to  $(X \cap \s_\e) \setminus V$ is
the usual Milnor fibration $\phi = \frac{f}{|f|}$, 
while its restriction  to the Milnor
tube $f^{-1}(\partial \D _\eta) \cap \mathbb{B}_\e$ is the
Milnor-L\^e fibration of $f$. Each element of the pencil $X_\t$
meets transversally the boundary sphere $\s_\e = \partial \B_\e$,
and the intersection is the union of the link of $f$ and two
homeomorphic fibers of $\phi$ over antipodal points in the
circle. Furthermore, the space ${\tilde X}$ obtained by the real
blow up of the ideal $(Re(f), Im(f))$ is a fibre bundle over $\R
\P^1$ with the $X_\t$ as fibres. These
  constructions work also, to some extent, for real
analytic map-germs, and give us a clear picture    of the
differences, concerning Milnor fibrations, between real and
complex analytic singularities.
\end{abstract}


\end{frontmatter}


\section*{Introduction}
Milnor's fibration theorem is a key-stone in singularity theory.
This is a result about the topology of the fibres of analytic
functions  near their critical points.

Let $X$ be an analytic subset of an open neighbourhood  $U$ of the
origin $\0$ in $\C^n$.
Given $f\colon(X,\0) \to (\C,0)$ holomorphic with a critical point
at $\0 \in X$ (in the stratified sense), there are two equivalent
ways of defining its \emph{Milnor fibration}.

The first   was given  in \cite[Thm.~4.8]{Mi2} for   $X$ smooth
and, as we show below, extends for arbitrary $X$:
let $\B_\e$ be a closed ball of sufficiently small radius $\e$ around $\0 \in \C^n$
and let $\s_\e$ be its boundary sphere, let $L_X=X\cap\s_\e$ be the link
of $X$ and let $L_f= f^{-1}(0)\cap\s_\e$
be the link of $f$ in $X$. Then the fibration is:
\begin{equation}\label{M-fib}
\phi = \frac{f}{|f|}\colon L_X \setminus L_f \longrightarrow \s^1.
\end{equation}
  This fibration theorem, for
general $X$, is implicit in the work of L\^e D\~ung Tr\'ang
\cite{Le:VCCAS} and a weaker form of it is also given in
 \cite[Thm.~3.9]{Durfee:NAS}.

The second version of the fibration theorem also originates in
Milnor's book \cite{Mi2}. For this, choose $\e >> \eta
>0$ sufficiently small and consider the {\it Milnor tube}
$$ N(\e,\eta) = X \cap \B_\e \cap f^{-1}(\partial \D_\eta) \,,$$
where $\D_\eta \subset \C$ is the disc of radius $\eta$ around $0
\in \C$. Then
\begin{equation}\label{ML-fib}
f \colon N(\e,\eta)  \longrightarrow \partial \D_\eta\,,
\end{equation}
is a fibre bundle, isomorphic to the previous bundle
(\ref{M-fib}).\footnote{Throughout this article we speak of
``equivalence'' of these (and similar) fibrations. This statement
must be made precise, since $N(\e,\eta)$ is compact and $L_X
\setminus L_f$ is not; the second fibration must be restricted to
$L_X $ minus an open regular neighbourhood of $L_f$ to have an
actual equivalence, but this determines the whole fibration on
$L_X \setminus L_f$ (see the proof of Theorem~\ref{Thm:fib.thm} in
Section \ref{Section-Proof of Theorem} below). The ``equivalence''
must be understood in this sense.} We notice that in his book, J.
Milnor proves only that for $X$ smooth, the fibres of
(\ref{ML-fib}) are isomorphic to those of (\ref{M-fib}), but he
does not prove that (\ref{ML-fib}) is actually a fibre bundle, a
fact he certainly knew when $f$ has an isolated critical point and
$X = \C^n$ (see \cite{Milnor:ISH}). H. Hamm in \cite{Hamm}
  extended Milnor's work to the case when $X\setminus f^{-1}(0)$ is
non singular, and  D.T. L\^e \cite[Thm.~(1.1)]{Le1} proved that
(\ref{ML-fib}) is a fibre bundle in full generality. We call
(\ref{ML-fib}) the Milnor-L\^e fibration, to distinguish it from
the equivalent fibration (\ref{M-fib}).

 In this work we improve, or refine, these
 fibration theorems in five directions, given by Theorems~\ref{Thm:Can.Dec},
\ref{Thm:fib.thm}, \ref{thm.spherefication}, \ref{Thm:blow.up} and
 \ref{theorem.real}.

The starting point, that originates in \cite {Se1, RS}, is to
notice that every holomorphic map $f$ as above determines a
canonical pencil of real analytic hypersurfaces, with axis $\,V =
f^{-1}(0)\,$, and this pencil gives rise to both fibrations
(\ref{M-fib}) and (\ref{ML-fib}) as we explain below.

For simplicity, denote also by $X$ the intersection $X \cap
\B_\e$. (See Section~\ref{sec:w.s} for the definition of
Whitney-strong stratifications.)

\begin{Thm}[\bf Canonical Decomposition]\label{Thm:Can.Dec}
For each $\t \in [0,\pi)$, let $\cL_\t$ be the line through $0$ in
$\R^2$ with an angle $\t$ (with respect to the $x$-axis). Set $\,V
= f^{-1}(0)\,$ and $X_\t =f^{-1}(\cL_\t)$. Then one has:
\begin{enumerate}[\bf i)]
\item
The $X_\t$ are all homeomorphic real analytic hypersurfaces of $X$
with singular set $\sing(V) \cup (X_\t \cap \sing(X))$. Their
union is the whole space $X$ and they all meet at $V$, which
divides each $X_\t$ in two homeomorphic halves.\label{it:decomposition}

\item
 If $\{S_\a\}$ is a Whitney stratification of $X$ adapted to
$V$, then the intersection of the strata with each $X_\t$
determines a Whitney-strong stratification of $X_\t$, and   for each
stratum $S_\a$ and each $X_\t$, the intersection $S_\a \cap X_\t$
meets transversally every sphere in $\B_\e$ centred at $\0$.\label{it:stratification}

\item
There is a uniform conical structure for all $X_\t$, {\it i.e.}, there is a (rugose) homeomorphism
\begin{equation*}
 h\colon (X \cap \B_\e, V \cap \B_\e) \to \bigl(\hbox{Cone}(L_X), \hbox{Cone}(L_f)\bigr),
\end{equation*}
which restricted to each $X_\t$ defines a homeomorphism
\begin{equation*}
 (X_\t \cap \B_\e) \cong \hbox{Cone}(X_\t \cap \s_\e).
\end{equation*}\label{it:cones}
\end{enumerate}
\end{Thm}

\vskip.2cm

\begin{Thm}[\bf Fibration Theorem]\label{Thm:fib.thm}
One has a commutative diagram of  fibre bundles
\begin{equation*}
\xymatrix{
(X \cap \B_\e) \setminus V\ar[r]^-{\Phi}\ar[rd]_{\Psi}  &\s^1\ar[d]^{\pi}\\
& \RP^1
}
\end{equation*}
where $\Psi(x) = (\hbox{Re}(f(x)):\hbox{Im}(f(x)))$ with fibre
$(X_\t\cap \B_\e)\setminus V$, $\Phi(x)=\frac{f(x)}{\norm{f(x)}}$
and $\pi$ is the natural two-fold covering. The restriction of
$\Phi$ to the link $L_X \setminus L_f $ is the usual Milnor
fibration $\phi$ in (\ref{M-fib}), while the restriction of $\Phi$ to the
Milnor tube $f^{-1}(\partial \D_\eta) \cap \B_\e$ is the
Milnor-L\^e fibration (\ref{ML-fib})   (up to multiplication by a
constant), and these two fibrations are equivalent.
\end{Thm}

\vskip.2cm

To prove Theorem~\ref{Thm:fib.thm} we introduce in
Section~\ref{subsec:spherification} the \emph{spherefication} of
$f$, which is an auxiliary function  defined by
$\mathfrak{F}(x)=\norm{x}\Phi(x)= \norm{x}
\,\frac{f(x)}{\norm{f(x)}}$. This map has the property that its
``Milnor tubes'' are precisely the ``Milnor fibrations on the
spheres''. More precisely we have the following Fibration Theorem.

\begin{Thm}\label{thm.spherefication}
 For $\e >0$ sufficiently small, one has a fibre bundle
\begin{equation*}
 \mathfrak{F} \colon \big( (X \cap \B_\e) \setminus V \big)  \longrightarrow
(\D_\e \setminus \{0\})\,,
\end{equation*}
taking $x$ into $\norm{x} \,\frac{f(x)}{\norm{f(x)}}$, where
$\D_\e$ is the disc in $\R^2$ centred at $0$ with radius $\e$.
Furthermore, the restriction of $\mathfrak{F}$ to each sphere
around $\0$ of radius $\e' \le \e$ is a fibre bundle over the
corresponding circle of radius $\e'$, and this  is the Milnor
fibration $\phi$ in \eqref{M-fib} up to multiplication by a
constant.
\end{Thm}

Our proofs actually show:

\vskip .2cm

\begin{Cor}\label{Cor:Mil.fib}
Let $f\colon (X,\0) \to (\C,0)$ be as above, a holomorphic map
with a critical point at $\0 \in X$, and consider its Milnor
fibration
$$\phi = \frac{f}{|f|}\colon L_X \setminus L_f \longrightarrow
\s^1 \,.$$ If the germ $(X,\0)$ is irreducible, then we have that
every pair of fibres of $\phi$ over antipodal points of $\s^1$ are
glued together along the link $L_f$ producing the link of a real
analytic hypersurface $X_\t$, which is homeomorphic to the link of
$\{Re \,f = 0\,\}$. Moreover, if both $X$ and $f$ have an isolated
singularity at $\0$, then   this homeomorphism is in fact a
diffeomorphism and the link of each $X_\t$ is diffeomorphic to the
double of the Milnor fibre of $f$ regarded as a smooth manifold
with boundary $L_f$.
\end{Cor}

The above  hypothesis of $X$ being irreducible can be relaxed,
assuming only that $X$ and $V$ do not share an irreducible
component at the origin (see Remark \ref{example-Jawad}).

Notice that for $\t =\pi/2$ the corresponding variety $X_\t$ is
$\{Re \, f = 0 \}$. Thus for instance, for the map $(z_1,z_2)
\buildrel{f} \over {\mapsto} z_1^2 + z_2^q$ one gets that the link
of $Re \,f$ is a closed, oriented surface  in the 3-sphere, union
of the Milnor fibres over the points $\pm i$; an easy computation
shows that it has genus $q-1$. It would be interesting to study
geometric and topological properties of the 4-manifolds one gets
in this way, by considering the link of the  hypersurface in
$\C^3$ defined by the real part of a holomorphic function with an
isolated critical point. For example, for the map $(z_1,z_2,z_3)
\buildrel{f} \over {\mapsto} z_1^2 + z_2^3 + z_3^5$, the
corresponding 4-manifold is the double of the famous $E_8$
manifold with boundary Poincare's homology 3-sphere.

\vskip.2cm

In order to complete the proof of Theorem~\ref{Thm:Can.Dec} we actually show (Section 3):

\begin{Thm}\label{Thm:blow.up}
Let ${\tilde X}$ be the   space obtained by the real blow-up of
$V$, {\it i.e.}, the blow-up of $(Re(f), Im(f))$. The projection
${\tilde \Psi}: {\tilde X} \rightarrow \R\P^1$ is a topological
 fibre bundle with fibre $X_\t$.
\end{Thm}

This  result  strengthens Theorem \ref{Thm:fib.thm} and implies
that all the hypersurfaces $X_\t$ are homeomorphic: Are they
actually analytically equivalent? we do not know the answer. The
proof of Theorem \ref{Thm:blow.up}   (in Section
\ref{section-Blow-up}) can be refined to prove also
Theorem~\ref{Thm:fib.thm}. However we prefer to give (in
Section~\ref{section-Fibration}) a direct proof of
Theorem~\ref{Thm:fib.thm}, which is elementary and
 throws  light into  the understanding of Milnor-type
fibrations for  real analytic maps, that we envisage in Section
\ref{section-Real-case}: We consider   real analytic map-germs
from $\R^{n+2}$ into  $\R^{2}$. One has a canonical pencil
$(X_{\t})$ associated to $f$ as in the holomorphic case, but now
the pencil may not have the uniform conical structure of Theorem
\ref{Thm:Can.Dec}. If it does, at least away from $V$, then we say
that $f$ is $d$-regular; we give various examples of families of
such singularities.
 We prove:

\begin{Thm}\label{theorem.real} Let $f:(U,\0) \to (\R^{2},0)$ be a locally surjective
real analytic map
  with an isolated critical value at $0 \in \R^2$, $U$ an open
neighbourhood of $\, \0$ in $\R^{n+2}$ and $V=f^{-1}(0)$. Assume
further that at $\0$, $f$ has the Thom $a_f$ property, it is
d-regular and  $\hbox{dim} \,V >0$. Then:

 \begin{enumerate}[i)]

\item One has a Milnor-L\^e fibration (a fibre bundle)
$$
f \colon N(\e,\eta)  \longrightarrow \partial \D_\eta\,,
$$
where $ N(\e,\eta) = \B_\e \cap f^{-1}(\partial \D_\eta) \,,$ is a
Milnor tube for $f$; $\D_\eta \subset \R^2$ is the disc of radius
$\eta$ around $0 \in \R^2$, $\e >> \eta > 0$. In fact this same
statement holds for $   \B_\e \cap f^{-1}( \D_\eta \setminus {0})
\,,$ which fibres over $\D_\eta \setminus
{0}$.\label{it:Milnor.Le}

\item
For
every sufficiently small $\e >0$  one has a commutative diagram of
fibre bundles,
\begin{equation*}
\xymatrix{
\B_\e \setminus V\ar[r]^-{\Phi}\ar[rd]_{\Psi}  &\s^1\ar[d]^{\pi}\\
& \RP^1
}
\end{equation*}
which induces by restriction,   a   fibre bundle  $\s_\e \setminus K_\e \buildrel{\phi}
\over \rightarrow \s^1$ with $\phi = f/\vert {f} \vert $ and $K_\e=V\cap \s_\e$.\label{it:Strong.Milnor}

\item The two fibrations above, one on the Milnor tube, the other on the sphere,
are equivalent.\label{it:equivalence}
\end{enumerate}
\end{Thm}

We remark that in the holomorphic case, critical values must be
isolated. This is not the case in general  for real analytic maps into $\R^2$
(cf. for instance \cite{Pichon-Seade}). We remark also that if in
Theorem~\ref{theorem.real} we remove the hypothesis that the image of $f$ is a neighbourhood of $0\in\R^2$,
then the results are true restricting the fibrations to their images.
This is carefully discussed in \cite{CSS:MR} where an example is given in \cite[Remark~4.1]{CSS:MR}.

The study of Milnor-type fibrations for real analytic
mappings is a subject   that goes   back to
Milnor's   work in \cite {Milnor:ISH, Mi2}, and there are several recent papers on the
topic.  We refer to
  \cite {Se2, Se3} for overviews on this, and to \cite{Cisneros, Massey, Oka,
  Pichon-Seade,  RS, Raimundo1, Raimundo-Tibar} for more recent work. Notice also that
  part of the content of Section~\ref{section-Real-case} generalises   to real analytic map-germs
  $(X,\0) \to (\R^k,0)$ with $k \ge 2$ and X singular.


\vskip.3cm

The authors are grateful to professors Bernard Teissier, L\^e
D\~ung Tr\'ang and David Trotman  for helpful conversations while
preparing this article, and to  Mihai Tib\v ar for pointing out a
couple of gaps in the original manuscript, and helping us to fill
up one of these gaps. The authors are also grateful to the referee for several important comments
The hospitality and support of the Abdus
Salam ICTP (from Trieste, Italy) are acknowledged with thanks by
the three authors. The first author also thanks Luis Hern\'andez
Hern\'andez (from the University of Sheffield, U. K.) for providing
him the right atmosphere while writing part of this article.

\section{Preliminaries}

This section contains well-know material about stratified analytic spaces.

\subsection{Stratifications and Whitney conditions}
Let $\F$ be either the real or the complex numbers. A
\emph{stratification} of a subset $X$ of  $\F^n$ is a
{\it locally finite} partition
$\{S_\alpha\}$ of $X$ into smooth, connected submanifolds of
$\F^n$ (called \emph{strata}) which satisfy the {\it
Frontier Condition}, that if
  $S_\alpha$ and $S_\beta$ are strata with
$S_\alpha\cap \bar{S}_\beta\neq\emptyset$, then $S_\alpha\subset \bar{S}_\beta$.

When $X$ is complex analytic, we say that a stratification
$\{S_\alpha\}$ is {\it complex analytic} if all the strata are
analytic. In the real analytic case, we say that $\{S_\alpha\}$ is
{\it subanalytic} if each $S_\a$, its closure $\bar S_\a$ and
$\bar S_\a \setminus S_\a$ are subanalytic.

\medskip

Now
consider a triple $(y,S_\alpha,S_\beta)$, where $S_\alpha$ and
$S_\beta$ are strata of $X$  with $y\in
S_\alpha\subset\bar{S}_\beta$. We say that the
 triple $(y,S_\alpha,S_\beta)$  is \emph{Whitney regular} if it
 satisfies the \emph{Whitney $(b)$ condition}:

\begin{enumerate}[i)]
\item given a sequence  $\{x_n\}\subset S_\beta$ converging in
$\F^n$ to $y\in S_\alpha$ such that the sequence of tangent spaces
$T_{x_n}S_\beta$ converges to a subspace $T\subset \F^n$ ; and
\label{i:planes}

\item a sequence $\{y_n\}\subset S_\alpha$  converging to $y\in
S_\alpha$ such that the sequence of lines (secants) $l_{x_iy_i}$
from $x_i$ to $y_i$ converges to a line $l$;\label{ii:lines}
\end{enumerate}
then one has $l\subset T$.

\medskip

By convergence of
tangent spaces or secants we mean convergence of the translates to the origin
of these spaces, so these are points in the corresponding
Grassmannian.

\medskip

There is also a \emph{Whitney $(a)$ condition}: in the above
situation \ref{i:planes}) one has that $T$ contains the space tangent to
$S_\alpha$ at $y$. It is an exercise to show that condition $(b)$
implies condition $(a)$.

\begin{definition} The stratification $\{S_\alpha\}$ of $X$ is
\emph{Whitney regular} (also called a \emph{Whitney
stratification}) if every triple $(y,S_\alpha,S_\beta)$ as above,
is Whitney regular.
\end{definition}

The existence of Whitney stratifications for every analytic space
$X$ was proved by Whitney in \cite[Thm.~19.2]{Whitney:TAV} for
complex varieties, and  by Hironaka \cite{Hironaka:SubAnal} in the general
setting.

\subsection{Whitney-strong stratifications}
\label{sec:w.s}

We now describe another regularity condition, defined by Verdier
in \cite{Ver}, improving Kuo's  regularity condition in
\cite{Kuo:RTAWS}. For  this let $A$ and $B$ be vector subspaces of
$\F^n$, and let $\pi_B$ be the orthogonal projection onto $B$. We
define the \emph{distance} (or \emph{angle}) between $A$ and $B$
by
\begin{equation}\label{eq:angle.def}
\delta(A,B)=\sup_{\substack{a\in A,\\\norm{a}=1}} \dist(a,B)
=\sup_{\substack{a\in A,\\\norm{a}=1}}\norm{a-\pi_B(a)}.\\[5pt]
\end{equation}

Notice that this is not symmetric in $A$ and $B$. Also,
$\delta(A,B)=0$ if and only if $A\subseteq B$, and
$\delta(A,B)=1$ if and only if there exists $a\in A$ such that
$a\perp B$. If $B$ is a subspace of $C$, then $\delta(A,C)\leq\delta(A,B)$.

The {\it Kuo-Verdier $(w)$ condition (also known as Whitney-strong
condition)} for a triple $(y,S_\alpha,S_\beta)$  as above is that
there exists a neighbourhood $\mathcal{U}_y$ of $y\in S_\alpha$ in
$\F^n$ and a constant $C>0$ such that
\begin{equation*}
\delta(T_{y'}S_\alpha,T_xS_\beta)\leq C \,\norm{y'-x}
\end{equation*}
for all $y'\in \mathcal{U}_y\cap S_\alpha$ and all $x\in \mathcal{U}_y\cap S_\beta$.

\begin{remark}\label{rem:W=WS}
Condition $(w)$ reinforces condition $(a)$, and for analytic
stratifications condition $(w)$ implies condition $(b)$
(\cite[Thm.~(1.5)]{Ver} or \cite{Kuo:RTAWS}). Moreover, for
complex analytic
 stratifications, Teissier  proved in \cite{Teissier:VarPolII} that
conditions $(b)$ and $(w)$ are actually equivalent.
\end{remark}

\begin{definition}
The stratification $\{S_\alpha\}$ is said to be \emph{Whitney-strong} if
every triple $(y,S_\alpha,S_\beta)$ satisfies condition
$(w)$.
\end{definition}

The existence of Whitney-strong stratifications for all analytic
spaces was proved by Verdier  in \cite[Thm.~(2.2)]{Ver}, using
Hironaka's theorem of resolution of singularities. There are
simpler proofs  in
\cite{Lojasiewicz-Stasica-Wachta:SSACV,Denkowska-Wachta:CSSACw}.

The following theorem summarises results from \cite{Whitney:TAV}
and \cite {Ver}.

\begin{theorem}\label{rmk:strat}
Let $(X,\0)$ be a (real or complex) analytic germ in $\F^n$, and
let $V$ be an analytic variety  in X. Then we can endow $X$ with a
locally finite Whitney-strong analytic stratification such that
$V$ is union of strata, $X\setminus (V\cup \sing(X)$ is a stratum
(possibly disconnected in the real case). We can further assume
$\0$ is a stratum and there exists a sufficiently small ball $\B$
around $\0 \in \F^n$, such that each stratum contains $\0$ in its
closure and is transverse to all the spheres in $\B$ centred at
$\0$.
\end{theorem}

This yields (compare with   \cite[Thm.~10.2]{Mi2} and \cite[Lemma~3.2]{BV}):

\begin{theorem}\label{conical structure}
 For  $\B$ as above, one has that the triple $(\B,X\cap \B, V \cap \B)$ is homeomorphic to
the cone over the triple $(\partial \B, X\cap \partial \B, V \cap\partial \B)$.
\end{theorem}

\begin{definition}
 A Whitney-strong stratification of $X$ as in Theorem~\ref{rmk:strat} will be called a
 \emph{Whitney stratification adapted to $V$}.
\end{definition}

\subsection{The Thom Property}
We now look at regularity conditions for the space $X$ relative to a
function on it. This originates in the work of R. Thom \cite{Thom:EMS}.

Let $X\subset U\subset\C^n$ be a complex analytic  subspace of an open set
 $U$ of $\C^n$, and let $f\colon (X,\0) \to (\C,0)$ be a holomorphic
function. Let
$\{S_\alpha\}_{\alpha\in A}$ be an analytic Whitney stratification of $X$
 and let $\{T_\gamma\}_{\gamma\in G}$ be an analytic Whitney stratification
of $\C$. We say that the stratifications $\{S_\alpha\}_{\alpha\in A}$ and
$\{T_\gamma\}_{\gamma\in G}$ give a \emph{Whitney stratification of $f$} if for
every $\alpha\in A$ there exists $\gamma\in G$ such that $f$ induces a
submersion from $S_\alpha$ to $T_\gamma$.

Consider the triple $(y,S_\alpha,S_\beta)$, with $y\in
S_\alpha\subset\bar{S}_\beta$. Set $f^\alpha_y
=f|_{S_\alpha}^{-1}(f(y))$, the fibre of $f|_{S_\alpha}$ which
contains $y$.

 We say that a triple $(y,S_\alpha,S_\beta)$ satisfies the
{\it Thom $(a_f)$ condition} if for every sequence  $\{x_n\}\subset
S_\beta$ converging in $\C^n$ to $y\in S_\alpha$,   one has (when
the limit exists):
\begin{equation*}
\lim_{n\to+\infty}\delta(T_yf^\alpha_y,T_{x_n}f^\beta_{x_n})=0.
\end{equation*}
The triple satisfies the {\it Strict Thom  $(w_f)$ condition} if
there exists a neighbourhood $\mathcal{V}_{y}$ of $y\in S_\alpha$
in $\C^n$ and a constant $D>0$ such that for all $y'\in
\mathcal{V}_{y}\cap S_\alpha$ and $x\in \mathcal{V}_{y}\cap
S_\beta$,
\begin{equation*}
\delta(T_{y'}f^\alpha_{y'},T_xf^\beta_x)\leq D \norm{y'-x}\,.
\end{equation*}

\begin{definition}
We say that the stratification satisfies \emph{Thom's} $(a_f)$
\emph{condition} (respectively $(w_f)$ \emph{condition}) if every
triple $(y,S_\alpha,S_\beta)$ satisfies
 condition $(a_f)$ (respectively $(w_f)$).
\end{definition}

\begin{definition} We say that the map $f$ on $X$ has the \emph{Thom property}
(respectively the \emph{strict Thom property}) if there is a
Whitney stratification of $f$ that  satisfies \emph{Thom's}
$(a_f)$ \emph{condition} (respectively the $(w_f)$ {condition}).
\end{definition}

Thom property for complex analytic maps was proved by Hironaka in
\cite[\S5 Cor.~1]{Hironaka:SF}. The $(w_f)$ property was proved in
\cite{Par} for the case $X$ smooth (and complex) and in
\cite[Thm.~4.3.2]{Brianson-Maisonobe-Merle:LSDSWCT} for the
general complex analytic case. The corresponding statement is false in general for real analytic maps.

\begin{remark}\label{rmk:all.reg.strat}
Let $X\subset U\subset\C^n$ be a complex analytic subset of an
open set $U$ of $\C^n$. Let $f\colon (X,\0) \to (\C,0)$ be
holomorphic. Let $\{S_\alpha\}$
 and $\{T_\gamma\}$ be complex analytic Whitney stratifications of $X$ and
$\C$ respectively, which define a stratification of $f$. By
\cite[Rmk.~4.1.2, Thm.~4.2.1,
Thm.~4.3.2]{Brianson-Maisonobe-Merle:LSDSWCT} one has that this
stratification of $f$ satisfies conditions $(a_f)$ and $(w_f)$.
\end{remark}

\subsection{Stratified rugose vector fields}

Let $X \subset \F^n$ be equipped  with a stratification
$\{S_\alpha\}$ and $f\colon X\to \R$ an analytic function. The
function $f$ is a \emph{rugose function} if for every stratum
$S_\alpha$ the restriction $f|_{S_\alpha}$ is of class $C^\infty$
and if for every $y\in S_\alpha$ there exists a neighbourhood $V$
of $y$ and a constant $C\in\R^*_+$, such that, for every $y'\in
V\cap S_\a$ and every $x\in V\cap X$ we have
\begin{equation*}
 \abs{f(y')-f(x)}\leq C\norm{y'-x}.
\end{equation*}
Notice that a rugose function is continuous. A vector valued map is \emph{rugose}
if every coordinate function is rugose.
A vector bundle $F$ on $X$ is \emph{rugose} if the changes of charts are
rugose.

\begin{definition}
 A \emph{rugose vector bundle $F$ on $X$ tangent to the stratification $\{S_\alpha\}$}
is a vector bundle on $X$ such that, for every stratum $S_\alpha$ there is an injection
$i_\alpha\colon F|_{S_\alpha}\hookrightarrow TS_\alpha$ and if $i\colon X\to\F^n$ is
the inclusion, the vector bundle morphism $F\to i^*T\F^n|_X$ induced by $i$ and the
$i_\alpha$ is rugose.
\end{definition}

A \emph{stratified vector field} $v$ on $X$ is a
section of the tangent bundle $T \F^n|_X$, such that at each  $x
\in X$, the vector $v(x)$ is tangent to the stratum that contains $x$.

A stratified vector field $v$ is called \emph{rugose} near $y\in
S_\a$, where $S_\a$ is a stratum of $X$, when there exists a
neighbourhood $\mathcal{W}_{y}$ of $y$ in $\F^n$ and a constant
$K>0$, such that
\begin{equation}\label{eq:rugose}
\norm{v(y')-v(x)}\leq K\norm{y'-x},
\end{equation}
for every $y'\in \mathcal{W}_{y}\cap S_\a$ and every
$x\in\mathcal{W}_{y}\cap S_\beta$, with $S_\a\subset\bar{S}_\beta$.

Rugose vector fields play a key-role in Verdier's proof of the
Thom-Mather isotopy theorems (see \cite[\S4]{Ver}).

\medskip
The following result of Verdier is important for this article. We
include a slight modification of the proof given in
\cite[Prop.~(4.6)]{Ver} which is more suitable for our purposes.

\begin{proposition}[{\cite[Prop.~(4.6)]{Ver}}]\label{rugose lifting}
Let $X$ be a real analytic space and $A$ a locally closed
subset of $X$ which is union of strata for some Whitney-strong stratification
$\{S_\a\}$. Let $Y$ be a non-singular real analytic space, $g\colon X \to Y$ a
real analytic map whose restriction to each stratum is a submersion,
 and $\eta$ a $C^\infty$ vector field on $Y$. Then  there exists a rugose stratified
 vector field $\xi$ on $A$ that lifts $\eta$, {\it i.e.}, for each $x \in A$
 one has $dg(\xi(x)) = \eta(g(x))$.
\end{proposition}

\begin{proof}
 Using a rugose partition of unity the problem is local on $A$. Let $x\in A$,
by \cite[Cor.~(4.5)]{Ver} there exists a neighbourhood $V$ of $x$
and a rugose vector bundle $F$ on $V\cap A$ tangent to the stratification $\{S_\alpha\cap V\}$,
which induces on the stratum $S_x\cap V$ containing $x$ the tangent bundle $TS_x$.
The map $f$ induces a rugose vector bundle morphism $Tf\colon F\to f^*TY$.
Since $f$ is a submersion on each stratum, $Tf$ is surjective in $x$, and therefore surjective
on a neighbourhood of $x$. Let $K$ be the kernel of $F$ which is the \emph{vertical} subbundle
of $F$. Let $H$ be a rugose \emph{horizontal} subbundle of $F$, that is, such that $F=K\oplus H$.
One way to construct such $H$ is to endow $F$ with a rugose metric and take $H=K^\perp$.
Hence $Tf$ induces a rugose isomorphism from $H$ to $f^*TY$. The inverse image of $f^*\eta$
by this isomorphism is a desired rugose stratified vector field.
\end{proof}

\begin{remark}\label{rmk:horizontal}
 Notice that in the proof of Proposition~\ref{rugose lifting}, different choices of
horizontal subbundles $H$ of $F$ give rise to different liftings of $\eta$.
\end{remark}


\section{The uniform conical structure}\label{section-Conical-structure}

Let $\inpr{\,\cdot\,}{\,\cdot\,}$ be the Hermitian inner product
on $\C^n$. We consider $\C^n$ as a $2n$-dimensional Euclidean real
vector space defining the Euclidean inner product as the real part
$\Re\inpr{\,\cdot\,}{\,\cdot\,}$. Throughout this article,  $X$ is
a complex analytic subset of an open set $U$ around the origin
$\0$ of $\C^n$. Let $f\colon (X,\0) \to (\C,0)$ be a non-constant
holomorphic map and we equip $X$ with a Whitney stratification
adapted to $V=f^{-1}(0)$.
 We let $\B$ be a ball centred at $\0$ of sufficiently small
 radius, so that $f$ has no critical points in $\B \setminus V$,
 each stratum in $X \cap \B$ has $\0$ in its
 closure and meets transversally every sphere in $\B$ centred at $\0$.

Recall that a point $x\in X\cap\B$ is a \emph{critical point of
$f$}, in the stratified sense, if $x$ is a critical point of $f$
restricted to the stratum which contains $x$  (see for instance
\cite[Part~1~\S2.1]{Goresky-MacPherson:SMT} for more on this subject).
That is, if $ \tilde f$ is an extension of $f$ to $U$, then the
kernel of $d \tilde f(x)$ contains the space tangent to the
stratum. One has the analogous definition in the real analytic
category (see \cite{Goresky-MacPherson:SMT}).

\subsection{The canonical pencil of a holomorphic map}
Given $f$, we associate to it a 1-parameter family of real valued
functions as follows.  For each $\t \in [0,\pi) \,$, consider the
 real line
$\mathcal L_{\theta} \subset \C$ passing through the origin with an angle $\t$
 with respect to the real axis, measured in the usual way.
Let $\mathcal L_{\theta}^\perp$ be the line orthogonal to
$\mathcal L_{\theta}$ and let
$\pi_{\t} \colon \C \to \cL_{\t}^\perp$ be the orthogonal projection.
 Set $h_{\t} = \pi_{\t} \circ f \,$, so that  $h_0$ and
$h_{\frac{\pi}{2}}$ are, respectively, the imaginary  and real parts of $f$.
Hence
$\{h_\t\}$ is a 1-parameter family of real analytic functions and if we set
$X_{\t} = h^{-1}_{\t}(0)$, then each $X_\t$ is a real hypersurface.

The first two lemmas below  are exercises and we leave the proofs
to the reader. They prove part of statement~{\bf\ref{it:decomposition})} of Theorem~\ref{Thm:Can.Dec}.

\begin{lemma}\label{lem:singXt}
The singular points of $X_\t$ are:
\begin{equation*}
\sing X_\t=\sing V\cup (X_\t\cap \sing X).
\end{equation*}
\end{lemma}

\begin{lemma}\label{decompositon}
One has $X\cap\B = \cup X_{\t}$ and
\begin{equation*}
 V = \cap X_{\t} =X_{\t_1} \cap X_{\t_2},
\end{equation*}
for each pair $\t_1 \ne \t_2 \pmod\pi$.
\end{lemma}

\begin{remark}\label{rem:gluing.links}
We notice  that each $X_\t$ is naturally the union of three sets:
the points $x \in X\cap\B$ such that $f(x)=0$, {\it i.e.}, $x\in
V$, and the points $x \in X\cap\B$ such that $f(x)$ is in one of
the two half lines of
 $ \cL_{\t} \setminus \{0\}$. Write this as:
\begin{equation*}
X_{\t} = E_{\t} \cup V \cup E_{\t+\pi} \,.
\end{equation*}
Similarly, if $\s=\partial\B$  one has:
\begin{equation}\label{eq:links}
(X_\t \cap \s) = (E_{\t} \cap \s) \, \cup \,
(V \cap \s) \, \cup \, (E_{\t+\pi} \cap \s) \, \, .
\end{equation}
\end{remark}

\begin{remark}
Let $\{S_{\alpha}\}$ be a Whitney stratification of $X$ adapted to $V$.
 Since $f\colon(X\cap \B)\setminus V\to \C\setminus\{0\}$ is a submersion on each stratum,
 then for each stratum $\{S_\a\}\subset X\setminus V$, the intersection $S_\a \cap X_\t $
 is a differentiable manifold of real codimension 1 in $\{S_\a\}$.
\end{remark}

The following lemma proves statement {\bf\ref{it:stratification})} of Theorem~\ref{Thm:Can.Dec}.

\begin{lemma}\label{it:St.trans.Sp}
There exists $\e_0>0$ such that given a Whitney stratification
$\{S_{\alpha}\}$ of $X$ adapted to $V$, the intersection of the strata with each $X_\t$
determines a Whitney strong stratification of $X_\t$ such that for each stratum $S_\a\neq\{0\}$ and each
$X_\t$ one has that $S_\a \cap X_\t $ meets transversally every
sphere in the ball $\B_{\e_0}$ centred at $\0$.
\end{lemma}

\begin{proof}
Since $\bar{S}_\a \cap X_\t = \overline{S_\a \cap X_\t}$, it is
clear that the partition of $X_\t$ induced by the intersection
with the strata $\{ S_\a\}$ satisfies the frontier condition and
defines a stratification of each $X_\t$. To see that this
stratification of $X_\t$ is Whitney-strong, first note that by
\cite[Rem.~(3.7)]{Ver}, $\{(X_\t\setminus V)\cap S_\a\}$ is a
Whitney-strong stratification of $X_\t\setminus V$, hence we just
need to check condition $(w)$ for triples $(y,S_\a,(S_\beta\cap
X_\t))$, with $S_\a\subset V$ and $(S_\beta\cap X_\t)\subset
X_\t\setminus V$. By Remark~\ref{rmk:all.reg.strat} $f$ satisfies
condition $(w_f)$, so there exists a neighbourhood
$\mathcal{V}_{y}$ of $y\in S_\alpha$ in $\C^n$ and a constant
$D>0$ such that for all $y'\in \mathcal{V}_{y}\cap S_\alpha$ and
$x\in \mathcal{V}_{y}\cap S_\beta$,
\begin{equation*}
\delta(T_{y'}S_\a,T_xf^\beta_x)\leq D \norm{y'-x}.
\end{equation*}
Since $T_xf^\beta_x\subset T_x(S_\beta\cap X_\t)$, for every $y'\in
\mathcal{V}_{y}\cap S_\alpha$ and $x\in \mathcal{V}_{y}\cap
(S_\beta\cap X_\t)$,
\begin{equation*}
\delta(T_{y'}S_\a,T_x(S_\beta\cap X_\t))\leq\delta(T_{y'}S_\alpha,T_xf^\beta_x)\leq D \norm{y'-x}.
\end{equation*}
Therefore the triple $(y,S_\a,(S_\beta\cap X_\t))$ satisfies
condition $(w)$. We claim that for each $\t$, each stratum $
S_\alpha \cap X_\t$ meets transversally every sufficiently small
sphere around $\0$. This is in fact an immediate consequence of
\cite[Lemma~2.4]{BLS}, which implies the existence of a continuous
vector field $v$ on $\B_{\e_0}\setminus V$ which is tangent to
each  $S_\a$, tangent to each $X_\t$ and transverse to every
sufficiently small sphere around $\0$.
\end{proof}

\subsection{The spherefication map}
\label{subsec:spherification}

We now introduce an auxiliary function associated to $f$, the
spherefication, which is very helpful for studying Milnor
fibrations.

As in Theorem~\ref{Thm:fib.thm}, define $\Phi\colon (X\cap
\B_\e)\setminus V\to\s^1$ by $\Phi(x)=\frac{f(x)}{\norm{f(x)}} $.
Define the real analytic map $\,\mathfrak{F}\colon(X\cap
\B_\e)\setminus V\to \C\setminus\{0\} \,,$ by
\begin{equation}\label{eq:def.g}
\mathfrak{F}(x)=\norm{x}\Phi(x)= \norm{x} \,\frac{
f(x)}{\norm{f(x)}}\,.
\end{equation}
Notice that from the definition we have:
\begin{equation}\label{eq:phi.f.F}
 \Phi=\frac{\mathfrak{F}(x)}{\norm{\mathfrak{F}(x)}}=\frac{f(x)}{\norm{f(x)}}.
\end{equation}
Also notice that given $z\in\C\setminus\{0\}$ with $\t=\arg z$, the
fibre $\mathfrak{F}^{-1}(z)$ is the intersection of $X_\t$ with the
sphere $\s_{\abs{z}}$ of radius $\abs{z}$ centred at $\0$, and
$\mathfrak{F}$ carries $\s_{\abs{z}} \setminus V$ into the circle
around  $0 \in \R^2$ of radius  $\abs{z}$. This motivates the
following definition:

\begin{definition} The analytic map $\mathfrak{F}$ is called the {\it  spherefication} of $f$.
\end{definition}

\begin{lemma}\label{lem:g.sub}
Let $\e_0>0$ as in Lemma~\ref{it:St.trans.Sp} and let $\{S_{\alpha}\}$ be a Whitney stratification adapted to $V$.
Then the   spherefication map $\mathfrak{F}$ is a submersion on each stratum.
\end{lemma}

\begin{proof}
Let $x\in S_\a$ and $\t=\arg f(x)$, then $x\in S_\a\cap X_\t$  and by
Lemma~\ref{it:St.trans.Sp}, $S_\a\cap X_\t$ is transverse to the sphere
$\s_{\norm{x}}$ of radius $\norm{x}$. Hence there is a vector
$\mu\in T_x(S_\a\cap\s_{\norm{x}})$ such that
$d_x \mathfrak{F}|_{S_\a}(\mu)$ is a non-zero vector in
$T_{\mathfrak{F}(x)}\s^1_{\norm{x}}$. On the other hand, since $S_\a$
and $\s_{\norm{x}}$ are transverse, there is a vector $\nu\in
T_xS_\a$ transverse to $\s_{\norm{x}}$, and since the fibre of
$\mathfrak{F}$ through $x$ is contained in $\s_{\norm{x}}$, we have
that $d_x \mathfrak{F}|_{S_\a}(\nu)$ is a non-zero vector transverse to
$\s^1_{\norm{x}}$.
\end{proof}

Now we prove statement~{\bf\ref{it:cones})} of Theorem~\ref{Thm:Can.Dec}.
For this we use:

\begin{proposition}\label{lem:uni.con.str}
Let $\{S_\alpha\}$ be a Whitney stratification of $X$ adapted to
$V$  and for each $\t \in [0,\pi)$ equip  $X_\t$ with
the stratification $\{X_\t\cap S_\a\}$  obtained by
intersecting $X_\t$ with the strata of $\{S_\alpha\}$. Then for every
sufficiently small ball $\B_\e$ around $\0$,
 there exists  a stratified, rugose vector field $v$ on
 $X\cap\B_\e$ which has the following properties:
\begin{enumerate}[i)]\setlength{\itemsep}{0pt}
\item
It is radial, {\it i.e.}, it is transverse to the
intersection of $X$ with all spheres in $\B_\e$ centred at
$\0$.\label{it:pr1}

\item It is tangent to the strata of each $X_\t$.\label{it:pr2}

\item $v(\0)=0$.\label{it:pr3}
\end{enumerate}
\end{proposition}

\begin{proof}
Let $u$ be the canonical radial vector field on $\C$ given by
$u(z)=z$. Using  $\mathfrak{F}$, by Proposition~\ref{rugose lifting}
we can lift $u$ to a stratified
rugose vector field $v_{\mathfrak{F}}$ on $(X\cap \B)\setminus V$, which satisfies
\begin{equation*}
d\mathfrak{F}_x(v_{\mathfrak{F}}(x))=u(\mathfrak{F}(x)),
\end{equation*}
for every $x\in(X\cap \B)\setminus V$, where $d$ is the
derivative. By the definition of $\mathfrak{F}$, the local flow
associated to $v_{\mathfrak{F}}$ is transverse to all spheres in
$\B_\e$ centred at $\0$, and by construction, its integral paths
move along $X_\t\setminus V$, i.e., it already satisfies
conditions \ref{it:pr1}) and \ref{it:pr2}) and it is rugose.

Recall that on $X$ we have a rugose vector field $v_{rad}$  which is
radial at $\0$, {\it i.e.}, it is tangent to each stratum and
transverse to all spheres around $0$, which gives the conical
structure of $X$ (\cite[Prop.~3.1.7]{Sch}). The idea to construct
the vector field $v$ satisfying the conditions of Proposition
\ref{lem:uni.con.str} is to  glue
the vector field $v_{\mathfrak{F}}$
on $(X\cap\B_\e)\setminus V$ with the vector field $v_{rad}$ on $X$ in
such a way that it keeps satisfying properties \ref{it:pr1}) and \ref{it:pr2}). The
problem is that   $v_{rad}$ may not be tangent to the strata of the
$X_\t$, so we must modify it  appropriately.  We will denote the
modified  vector field by $\tilde{v}_{rad}$. This will be a rugose
vector field defined on $X \cap B_\e$, which is radial on $V$,
tangent to each stratum $X_\t \cap S_\a$ and such that gluing
$\tilde{v}_{rad}$ and $v_{\mathfrak{F}}$ on $X\setminus V$ by a
rugose  partition of unity, we obtain a vector field $v$
with properties \ref{it:pr1}) to \ref{it:pr3}). This  is constructed as follows.

In \cite{Sch} the stratified radial vector field $v_{rad}$ is constructed by induction
 on the dimension of the strata \cite[\S1.5, Thm.~3.1.5]{Sch}, and it is shown (\cite[Prop. 3.1.7]{Sch})
that it can be assumed to be rugose. We  modify $v_{rad}$ to  have the desired property  at each stage.
To start the induction, define $\tilde{v}_{rad}(\0)=0$ to get property \ref{it:pr3}). Now
suppose that we have constructed $\tilde{v}_{rad}$ on the strata
of dimension less than $p$ and it is rugose, which is always possible by \cite[Prop. 3.1.7]{Sch}.

Let $S_\beta$ be a stratum of dimension $p$. Extend $\tilde{v}_{rad}$ to a radial
stratified rugose vector field $v_{rad}$ on $S_\beta$ as in \cite[Thm.~3.1.5]{Sch}.
If $S_\beta\subset V$, for every $x\in S_\beta$ we define $\tilde{v}_{rad}(x)=v_{rad}(x)$.
If $S_\beta\subset X\setminus V$, let $x\in S_\beta$ and denote by $f^\beta_x$ the fibre
 of $f|_{S_\beta}$ which contains $x$, {\it i.e.}, $f^\beta_x=f|_{S_\beta}^{-1}(f(x))$. Since
 $f(x)$ is a regular value of $f|_{S_\beta}$, $f^\beta_x$ is a differentiable submanifold
of $S_\beta$. Clearly $f^\beta_x\subset X_\t\cap S_\beta\subset
X_\t\setminus V$ with $\t=\arg f(x)$.
 Define the vector $\tilde{v}_{rad}(x)$ by projecting the vector $v_{rad}(x)$ to
 the tangent space $T_x f^\beta_x\subset T_x(X_\t\cap S_\beta)$.

We claim that $\tilde{v}_{rad}$ is also rugose. For this, let $S_\a$
be a stratum of dimension less than $p$ such that
$S_\a\subset\bar{S}_\beta$ and let $y\in S_\a$. Since the
stratification of $f$ satisfies condition $(w_f)$ (see Remark
\ref{rmk:all.reg.strat}), there exists a neighbourhood
$\mathcal{V}_{y}$
 of $y$ where the following inequality is satisfied.
\begin{equation}\label{eq:wf.in.alpha}
\delta(T_{y'}f^\alpha_{y'},T_xf^\beta_x)\leq D \norm{y'-x},
\end{equation}
for all $y'\in \mathcal{V}_{y}\cap S_\alpha$ and all
$x\in \mathcal{V}_{y}\cap S_\beta$.

On the other hand, since the vector field $v_{rad}$ is rugose,
there exists a neighbourhood $\mathcal{W}_{y}$ of $y$ where the
following inequality is satisfied
\begin{equation}\label{eq:rug.in.alpha}
\norm{v_{rad}(y')-v_{rad}(x)}
\leq K\norm{y'-x},
\end{equation}
for every $y'\in \mathcal{W}_{y}\cap S_\a$ and every
$x\in\mathcal{W}_{y}\cap S_\beta$.

Let $\mathcal{N}_{y}$ be an open ball around $y$ such that
$\mathcal{N}_{y}\subset\mathcal{V}_{y}\cap\mathcal{W}_{y}$ and set
\begin{equation*}
M=\sup_{y'\in\overline{\mathcal{N}_{y}\cap S_\a}}\norm{v_{rad}(y')}.
\end{equation*}
Let $y'\in\mathcal{N}_{y}\cap S_\a$ and  $x\in\mathcal{N}_{y}\cap
S_\beta$.

\vspace*{12pt}
\noindent\textbf{Case 1:} $S_\beta\subset V$.
\vspace*{12pt}

Since in $V$ the vector field $\tilde{v}_{rad}$ equals $v_{rad}$ and $v_{rad}$ is rugose,
by \eqref{eq:rug.in.alpha} we have
\begin{equation*}
\norm{\tilde{v}_{rad}(y')-\tilde{v}_{rad}(x)}=
\norm{v_{rad}(y')-v_{rad}(x)}\leq K\norm{y'-x}.
\end{equation*}
Hence $\tilde{v}_{rad}$ satisfies inequality \eqref{eq:rugose}
in $\mathcal{N}_{y}$.

\vspace*{12pt}
\noindent\textbf{Case 2:} $S_\beta\subset X\setminus V$.
\vspace*{12pt}

Let $\pi$ be the orthogonal projection of $\C^n$ onto $T_x f^\beta_x$.
By \eqref{eq:wf.in.alpha} and \eqref{eq:angle.def} we have that
\begin{equation*}
\biggl\Vert\frac{v_{rad}(y')}{\norm{v_{rad}(y')}}-\pi\biggl(\frac{v_{rad}(y')}
{\norm{v_{rad}(y')}}\biggr)\biggr\Vert\leq\delta(T_{y'}f^\alpha_{y'},T_xf^\beta_x)\leq D\norm{y'-x}.\\
\end{equation*}
Hence
\begin{equation}\label{eq:wf.ineq}
\norm{v_{rad}(y')-\pi(v_{rad}(y'))}\leq\norm{v_{rad}(y')}D\norm{y'-x}
\leq MD\norm{y'-x}.
\end{equation}
On the other hand, since $\pi$ is an orthogonal projection, by
\eqref{eq:rug.in.alpha} one has:
\begin{equation}\label{eq:rug.ineq}
\norm{\pi(v_{rad}(y'))-\pi(v_{rad}(x))}\leq \norm{v_{rad}(y')-v_{rad}(x)}\leq K\norm{y'-x}.
\end{equation}
Therefore using \eqref{eq:wf.ineq} and \eqref{eq:rug.ineq} we get:
\begin{align*}
\norm{\tilde{v}_{rad}(y')-\tilde{v}_{rad}(x)}&=\norm{v_{rad}(y')-\pi(v_{rad}(x))}\\
    &\leq \norm{v_{rad}(y')-\pi(v_{rad}(y'))}
    +\norm{\pi(v_{rad}(y'))-\pi(v_{rad}(x))}\\
    &\leq (MD+K)\norm{y'-x}.
\end{align*}
Hence $\tilde{v}_{rad}$ satisfies inequality \eqref{eq:rugose} in $\mathcal{N}_{y}$,
proving that $\tilde{v}_{rad}$ is rugose and therefore integrable \cite[Prop.~(4.8)]{Ver}.

Notice that when we modify the radial vector field $v_{rad}$ to obtain $\tilde{v}_{rad}$,
 it may happen that $\tilde{v}_{rad}(x)$ is no longer transverse to the sphere,
 this is the case if the fibre through $x$ is tangent to the
 sphere; it may even happen that $\tilde{v}_{rad}(x)$ vanishes.

Gluing $\tilde{v}_{rad}$ and $v_{\mathfrak{F}}$ on $X\setminus V$ by a rugose
 partition of unity we obtain a vector field $v$ defined on all of $\B_\e$
 with the desired properties.
\end{proof}

\begin{proof}[Proof of {\bf\ref{it:cones})} in Theorem~\ref{Thm:Can.Dec}]

Recall that $\inpr{\,\cdot\,}{\,\cdot\,}$ is the Hermitian inner product on $\C^n$ and that
we consider $\C^n$ as a $2n$-dimensional Euclidean real vector space defining
the Euclidean inner product as the real part $\Re\inpr{\,\cdot\,}{\,\cdot\,}$.
Let $r\colon\C^n\to\R$ be the real analytic function defined by $r(x)=\norm{x}^2$.
The \emph{real} gradient of $r$ is given by $\grad_\R r(x)=2x$, so the chain
rule for the derivative of $r$ along a path $x=p(t)$ takes the form
\begin{equation}\label{eq:ch.rule}
 \frac{dr(p(t))}{dt}=\Re\Bigl\langle\frac{dp}{dt},2x\Bigr\rangle.
\end{equation}
We proceed as in the proof of \cite[Thm.~2.10]{Mi2}: let $v$ be the
vector field constructed in Proposition~\ref{lem:uni.con.str}.
By property \ref{it:pr1}) $v$ is a radial
vector field pointing away from $\0$, thus we have that
\begin{equation*}
 \Re\inpr{v(x)}{x}>0.
\end{equation*}
 Normalise $v$ by setting
\begin{equation*}
 \hat{v}(x)=v(x)/\Re\inpr{v(x)}{2x}.
\end{equation*}
Let $x=p(t)$ be a solution of the differential equation
\begin{equation}\label{eq:diff.eq}
 \frac{dp(t)}{dt}=\hat{v}(p(t)).
\end{equation}
Then by \eqref{eq:ch.rule} we have that
\begin{equation*}
\frac{dr}{dt}(x)=\Re\inpr{\hat{v}(x)}{2x}=1.
\end{equation*}
Therefore $r(p(t))=t+\text{constant}$. Substracting a constant from the parameter $t$ if necessary, we may
suppose that
\begin{equation*}
 r(p(t))=\norm{p(t)}^2=t.
\end{equation*}
By \cite[p.~20]{Mi2}, the solution $p(t)$ can be extended through the interval $(0,\e^2]$.

For every $a\in X\cap\s_\e$ let $p(t)=P(a,t)$ be the unique solution of
\eqref{eq:diff.eq} which satisfies the initial condition $p(\e^2)=P(a,\e^2)=a$.
Clearly this function $P$ gives a \emph{rugose homeomorphism} from the product
$L_X\times(0,\e^2]$ onto $(X\setminus\{\0\})\cap\B_\e$.

Since $\hat{v}$ is a stratified vector field and $V$ is a finite union of strata,
any solution curve which touches a stratum $S_\a\subset V$ must be contained in $S_\a\subset V$. Therefore, $P$
is in fact a rugose homeomorphism of pairs from $(L_X\times(0,\e^2],L_f\times(0,\e^2])$ onto
$((X\setminus\{\0\})\cap\B_\e,(V\setminus\{\0\})\cap\B_\e)$.

Furthermore, since the vector field $v(x)$ is tangent to the stratum
$X_\t\cap S_\a$ of $X_\t\setminus V$ for all $x\in X\cap S_\a$,
every solution curve which touches $X_\t\cap S_\a$
must be contained in $X_\t\cap S_\a\subset X_\t\setminus V$.
Hence $P$ restricts to a rugose homeomorphism $(X_\t \cap \s_\e)\times(0,\e^2] \cong ((X_\t\setminus\{\0\}) \cap \B_\e)$ for
every $\t\in [0,\pi)$.

Finally, note that $P(a,t)$ tends uniformly to $\0$ as $t\to 0$. Therefore the correspondence
\begin{equation*}
 ta\to P(a,t),\quad 0<t\leq 1,\quad a\in X\cap\s_\e,
\end{equation*}
extends uniquely to a homeomorphism from $\hbox{Cone}(L_X)$ to $(X\cap\B_\e)$
and we arrive to statement {\bf\ref{it:cones})} in Theorem~\ref{Thm:Can.Dec}.
\end{proof}

\medskip

\begin{remark}
Notice that  if $X \setminus \0$ is non-singular,   the proof
above leads to a smooth vector field having properties \ref{it:pr1})-\ref{it:pr3}) in Proposition~\ref{lem:uni.con.str}.
This yields to
  a {\it diffeomorphism} between
 $X \setminus \0$ and the cylinder $L_X \times (-\infty,0]$,
where $L_X$ is the link of $X$, inducing for each $\t$ a
diffeomorphism $X_\t  \setminus \0 \cong L_{X_\t} \times (-\infty,0]$.
\end{remark}

\vskip 4pt

Let us denote by $\X_{(X\setminus V)\cap\B_\e }$ the decomposition of
$(X\setminus V)\cap\B_\e $ as the union of all $(X_\t \setminus V)\cap\B_\e$; similarly,
denote by $\X_{L_X \setminus L_f}$ the decomposition of $L_X \setminus L_f$ as the union
of all  $(X_\t\cap \partial \B_\e) \setminus L_f$.

\medskip

The following is an immediate consequence of the proof of
Proposition~\ref{lem:uni.con.str}.

\begin{corollary}  The decomposition $\X_{(X\setminus V)\cap\B_\e }$ is
homeomorphic to the cylinder $(\X_{L_X \setminus L_f}) \times
(-\infty,0]$.   Furthermore, if $X$ is non-singular away from $V$,
then the above homeomorphism can be taken to be a
 diffeomorphism.
\end{corollary}

\vskip 24pt


\section{The Fibration theorems}\label{section-Fibration}

In this section we prove Theorems~2 and 3. The proof of Theorem~3
is based on the results in Section \ref{section-Conical-structure}
together with Lemmas~\ref{lem:curve.strat} and
\ref{prop:arg.lambda} below, which are extensions of Milnor's
Lemmas 4.3 and 4.4 in \cite {Milnor:ISH} to the case of maps
defined on singular varieties.

Throughout this section let $\tilde{f}\colon U\to \C$ be an analytic extension of $f$ to
an open neighbourhood $U$ of $\0$ in $\C^n$.

\begin{lemma}\label{lem:curve.strat}
Let $S_\a$ be a stratum not contained in $V$, $x\in S_\a$ and let
$\pi_{\alpha_x}$ be the orthogonal projection of $\C^n$ onto
$T_xS_\a$. Let $p\colon[0,\e)\to S_\a$ be a real analytic path
with $p(0)=\0$ such that, for each $t>0$, the number
$f\bigl(p(t)\bigr)$ is non-zero and the vector
$\pi_{\alpha_{p(t)}}\bigl(\grad\log
\tilde{f}\bigl(p(t)\bigr)\bigr)$ is a complex multiple
$\lambda(t)\pi_{\alpha_{p(t)}}(p(t))$. Then the argument of the
complex number $\lambda(t)$ tends to zero as $t\to 0$. In other
words $\lambda(t)$ is non-zero for small positive values of $t$
and $\lim_{t\to 0}\frac{\lambda(t)}{\abs{\lambda(t)}}=1$.
\end{lemma}

\begin{proof}
We follow the proof of \cite[Lemma~4.4]{Milnor:ISH}.
 Consider the Taylor expansions
\begin{align*}
 p(t)&=\mathbf{a}t^k+\mathbf{a}_1t^{k+1}+\mathbf{a}_2t^{k+2}+\dots,\\
f\bigl(p(t)\bigr)&=bt^l+b_1t^{l+1}+b_2t^{l+2}+\dots,\\
\grad \tilde{f}\bigl(p(t)\bigr)&=\mathbf{c}t^m+\mathbf{c}_1t^{m+1}+\mathbf{c}_2t^{m+2}+\cdots,
\end{align*}
where the leading coefficients $\mathbf{a}$, $b$ and $\mathbf{c}$
are non-zero. (The identity
$\frac{df}{dt}=\inpr{\frac{dp}{dt}}{\grad \tilde{f}}$ shows that
$\grad \tilde{f}(p(t))$ cannot be identically zero.) The leading
exponents $k$, $l$ and $m$ are integers with $k\geq1$, $l\geq1$
and $m\geq0$. The series are all convergent say for $\abs{t}<\e'$.
To simplify notation, set $\pi_{\alpha_t}=\pi_{\alpha_{p(t)}}$.
For each $t>0$ we have
\begin{align*}
 \pi_{\alpha_t}\Bigl(\grad\log \tilde{f}\bigl(p(t)\bigr)\Bigr)&=\lambda(t)\pi_{\alpha_t}\bigl(p(t)\bigr),\\
 \pi_{\alpha_t}\Bigl(\grad\tilde{f}\bigl(p(t)\bigr)\Bigr)
&=\lambda(t)\pi_{\alpha_t}\bigl(p(t)\bigr)\bar{f}\bigl(p(t)\bigr),
\end{align*}
The vector $\pi_{\alpha_t}\Bigl(\grad
\tilde{f}\bigl(p(t)\bigr)\Bigr)$ is non-zero since $f$ is a
submersion on $S_\alpha$. Hence there exists a smallest $s$ such
that $\pi_{\alpha_t}(\mathbf{c}_s)\neq0$. On the other hand, the
vector $\pi_{\alpha_t}\bigl(p(t)\bigr)$ is non-zero since
$S_\alpha$ is transverse to all the spheres centred at $\0$ of
radius less than $\e$. Hence, there  exists a smallest $r$ such
that $\pi_{\alpha_t}(\mathbf{a}_r)\neq0$. Therefore
\begin{align*}
 \pi_{\alpha_t}(\mathbf{c}t^m+\mathbf{c}_1t^{m+1}+\cdots)
    &=\lambda(t)\pi_{\alpha_t}(\mathbf{a}t^k+\mathbf{a}_1t^{k+1}+\dots)(\bar{b}t^l+\bar{b}_1t^{l+1}+\dots)\\
\pi_{\alpha_t}(\mathbf{c}_s)t^{m+s}+\dots
    &=\lambda(t)\bigl(\pi_{\alpha_t}(\mathbf{a}_r)t^{k+r}+\dots\bigr)\bigl(\bar{b}t^l+\dots\bigr)\\
\pi_{\alpha_t}(\mathbf{c}_s)t^{m+s}+\dots
    &=\lambda(t)\bigl(\pi_{\alpha_t}(\mathbf{a}_r)\bar{b}t^{k+r+l}+\dots\bigr).
\end{align*}
Comparing corresponding components of these two vector valued
functions, we see that $\lambda(t)$ is a quotient of real analytic
functions, and therefore it has a Laurent expansion of the form
\begin{equation*}
 \lambda(t)=\lambda_0t^{m+s-k-r-l}(1+d_1t+d_2t^2+\dots).
\end{equation*}
Furthermore the leading coefficients must satisfy the equation
\begin{equation*}
\pi_{ \alpha_t}(\mathbf{c}_s)=\lambda_0\pi_{\alpha_t}(\mathbf{a}_r)\bar{b}.
\end{equation*}
Substituting this equation in the power series expansion of the identity
\begin{equation*}
 \frac{df(p(t))}{dt}=\Bigl\langle\frac{dp}{dt},\grad \tilde{f}\bigl(p(t)\bigr)\Bigr\rangle,
\end{equation*}
and noting that since $p(t)\in S_\a$, $t\in[0,\e)$, we have that $\frac{dp}{dt}\in T_{p(t)}S_\a$ and therefore
$\pi_{\alpha_t}(\frac{dp}{dt})=\frac{dp}{dt}$. Thus we obtain
\begin{equation*}
 \frac{df\bigl(p(t)\bigr)}{dt}
=\Bigl\langle\pi_{\alpha_t}\Bigl(\frac{dp}{dt}\Bigr),\pi_{\alpha_t}\Bigl(\grad \tilde{f}\bigl(p(t)\bigr)\Bigr)\Bigr\rangle.
\end{equation*}
Therefore
\begin{align*}
 (l bt^{l-1}+\dots)
&=\Bigl\langle\pi_{\alpha_t}(k\mathbf{a}t^{k-1}+\dots),\pi_{\alpha_t}(\mathbf{c}t^m+\dots)\Bigr\rangle,\\
 &=\Bigl\langle k\pi_{\alpha_t}(\mathbf{a}_r)t^{k+r-1}+\dots, \pi_{\alpha_t}(\mathbf{c}_s)t^{m+s}+\dots\Bigr\rangle\\
 &=\Bigl\langle k\pi_{\alpha_t}(\mathbf{a}_r)t^{k+r-1}+\dots,
\lambda_0\pi_{\alpha_t}(\mathbf{a}_r)\bar{b}t^{m+s}+\dots\Bigr\rangle\\
 &=k\norm{\pi_{\alpha_t}(\mathbf{a}_r)}^2\bar{\lambda_0}bt^{k+r+m+s-1}+\dots.
\end{align*}
Comparing the leading coefficients it follows that
\begin{equation*}
 l=k\norm{\pi_{\alpha_t}(\mathbf{a}_r)}^2\bar{\lambda_0}
\end{equation*}
which proves that $\lambda_0$ is a positive real number. Therefore
\begin{equation*}
 \lim_{t\to0}\arg\lambda(t)=0,
\end{equation*}
which completes the proof of Lemma~\ref{lem:curve.strat}.
\end{proof}

\begin{lemma}\label{prop:arg.lambda}
Let $X$ be an analytic subset of an open neighbourhood  $U$ of the
origin $\0$ in $\C^n$. Let $f\colon (X,\0) \to (\C,0)$ be
holomorphic  and set $V =f^{-1}(0)$. Let $\tilde{f}\colon U\to \C$
be an analytic extension of $f$ to $U$. Let $S_\a$ be a stratum
not contained in $V$, $x\in S_\a$ and let $\pi_{\alpha_x}$ be the
orthogonal projection of $\C^n$ onto $T_xS_\a$. There exists a
number $\e_0>0$ so that, for every $x\in S_\a$ with
$\norm{x}\leq\e_0$, the two vectors $\pi_{\alpha_x}(x)$ and
$\pi_{\alpha_x}(\grad\log \tilde{f}(x))$ are either linearly
independent over the complex numbers or else
\begin{equation*}
 \pi_{\alpha_x}\bigl(\grad\log\tilde{f}(x)\bigr)=\lambda\pi_{\alpha_x}(x)
\end{equation*}
where $\lambda$ is a non-zero complex number whose argument has absolute value less than $\pi/4$.
\end{lemma}

\begin{proof}
 The proof is the same as in \cite[Lemma~4.3]{Milnor:ISH} defining
\begin{equation*}
 W=\set{z\in S_\a}{\pi_{\alpha_z}\bigl(\grad\log \tilde{f}(z)\bigr)=\mu\pi_{\alpha_z}(z),\,\mu\in\C},
\end{equation*}
and using the \emph{analytic} curve selection lemma \cite[Prop.~2.2]{BV}.
\end{proof}

As in Milnor's case, we have the following corollary as an immediate consequence.

\begin{corollary}\label{cor:lin.in.R}
 Let $S_\a$ be a stratum not contained in $V$. For every $x\in S_\a$ which is sufficiently close to the origin, the
 two vectors $\pi_{\alpha_x}(x)$ and $\pi_{\alpha_x}\bigl(i\grad\log \tilde{f}(x)\bigr)$ are linearly independent over $\R$.
\end{corollary}

The vector $\pi_{\alpha_x}(x)$ is the normal vector in $S_\a$ of
the codimension 1 submanifold $\s_{\norm{x}}\cap S_\a$ and the
vector $\pi_{\alpha_x}\bigl(i\grad\log \tilde{f}(x)\bigr)$ is the
normal vector in $S_\a$ of the codimension 1 submanifold $S_\a\cap
X_\t$. Therefore Corollary~\ref{cor:lin.in.R} gives another proof
of the second statement of Theorem~\ref{Thm:Can.Dec}-{\bf\ref{it:stratification})}, 
that the intersection $S_\a \cap X_\t$ meets transversally every sphere in $\B_\e$ centred at $\0$.

\begin{proposition}\label{lem:vec.fld.sph}
There exists a complete, stratified, rugose, vector field $w$ on $(X\cap\B_\e)\setminus
V$, tangent to all the spheres in $\B_\e$ centred at $\0$, and whose
orbits are transverse to the $X_{\t} \setminus V$ and permute
them: for each fixed time $t$, the flow carries each $X_{\t}
\setminus V$ into $X_{\t + t} \setminus V$, where the angle $\t +
t$ must be taken modulo $\pi$. In particular, for $t = \pi$ the
flow interchanges the two halves of
 $X_{\t} \setminus V$.
\end{proposition}

\begin{proof}
Let $\e_0>0$ as in Lemma~\ref{it:St.trans.Sp}.
By Lemma~\ref{lem:g.sub}, the map
\begin{equation*}
\mathfrak{F} \colon(X\cap \B_\e)\setminus V\to \C\setminus\{0\}
\end{equation*}
is a submersion on each stratum $S_{\a}\subset\B_{\e_0}$.

Consider the vector field $\bar{w}$ on $\C\setminus\{0\}$ given by
$\bar{w}(z)=iz$, which is tangent to all the circles $\s^1_\eta\subset\C$.
By Proposition~\ref{rugose lifting} we can lift $\bar{w}$ using $\mathfrak{F}$, to a
stratified rugose vector field $w$ on $(X \cap \B_{\e_0})\setminus V$.
By \cite[Prop.~(4.8)]{Ver} this vector field is integrable and since
$d\mathfrak{F}_x(w(x))=\bar{w}(\mathfrak{F}(x))$, the integral curve $p(t)$ of $w$
is sent by $\mathfrak{F}$ to the circle $\s^1_{\norm{x}}\subset\C$ of radius $\norm{x}$. Thus,
$p(t)$ is transverse to $X_\t$ with $\t=\arg\mathfrak{F}(x)$. On the other hand, by the
definition of $\mathfrak{F}$, we have that $p(t)$ lies in the sphere $\s_{\norm{x}}\subset\C^n$ and therefore $w$
is tangent to all the spheres in $\B_{\e_0}$.
The solution $p(t)$ certainly exists locally and can be extended over some maximal open
interval of $\R$.

Since $(X\cap\B_\e)\setminus V$ is not compact, we have to guarantee that $p(t)$ cannot
tend to $V$ as $t$ tends to some finite limit $t_0$, that is, that $w$ is a complete vector field.
The lifting $w$ of the vector field $\bar{w}$ is not unique and it depends on the choice of a
horizontal subbundle $H$ of the rugose vector field $F$ tangent to the stratification $\{S_\alpha\}$
which induces the tangent bundle $TS_\a$ on a stratum $S_\a\cap V$(see Remark~\ref{rmk:horizontal}).
Hence we have to choose an appropriate horizontal subbundle $H$ to ensure that the lifting $w$
is complete. We do this as follows.

\medskip
Let $\tilde{f}\colon U\to \C$ be an analytic extension of $f$ to
an open neighbourhood $U$ of $\0$ in $\C^n$. Let
$\tilde{V}=\tilde{f}^{-1}(0)$, $\tilde{h}_\t=\pi_\t\circ\tilde{f}$
and define $\tilde{X}_\t=\tilde{h}^{-1}_\t(0)$.
Following \cite[\S4]{Mi2}, since $\tilde{f}=\abs{\tilde{f}(x)}e^{i\theta(x)}$,
we have that $\log \tilde{f}(x)=\log\abs{\tilde{f}(x)}+i\theta(x)$, thus
$\theta(x)=\Re(-i\log\tilde{f}(x))$. Differentiating along a curve $x=p(t)$ we have that
\begin{equation}\label{eq:d.theta}
 \frac{d\theta(p(t))}{dt}=\Re\Bigl\langle\frac{dp}{dt},i\grad\log \tilde{f}(x)\Bigr\rangle.
\end{equation}
Therefore, the vector $i\grad\log \tilde{f}(x)$ is the \emph{real} gradient of the real
analytic function $\theta$. Since the $\tilde{X}_\t$'s are the fibres of $\theta$ the vector
$i\grad\log \tilde{f}(x)$ is \emph{normal} to $\tilde{X}_{\t(x)}$.

On the other hand, $\log\abs{\tilde{f}(x)}=\Re(\log \tilde{f}(x))$ and differentiating along a curve $x=p(t)$ we have that
\begin{equation}
 \frac{d\log\abs{\tilde{f}(x)}}{dt}=\Re\Bigl\langle\frac{dp}{dt},\grad\log \tilde{f}(x)\Bigr\rangle.
\end{equation}
Hence the vector $\grad\log \tilde{f}(x)$ is the \emph{real} gradient of the real
analytic function $\log\abs{\tilde{f}(x)}$. Since the Milnor tubes are the fibres of $\log\abs{\tilde{f}(x)}$,
the vector $\grad\log \tilde{f}(x)$ is \emph{normal} to the Milnor tube $N(\e,\abs{\tilde{f}(x)})$.

We have that $X_\t=X\cap\tilde{X}_\t$. Let $S_\a$ be a stratum of $X$ not
contained in $V$; notice that $S_\a\cap X_\t = S_\a \cap\tilde{X}_\t$.
Consider $x\in S_\a\cap X_\t$ and let $\s_{\norm{x}}$ be the sphere of
radius $\norm{x}$ centred at $\0$.
By the definition of $\mathfrak{F}$ we have
that the fibre of $\mathfrak{F}|_{S_\alpha}$ which contains $x$ is given by
\begin{equation}\label{eq:fib.esph}
\mathfrak{F}|_{S_\alpha}^{-1}(\mathfrak{F}(x))=S_\alpha\cap \s_{\norm{x}}\cap X_\t.
\end{equation}
We denote this fibre by $\mathfrak{F}_x^\a$. To simplify notation
set $\s_{\norm{x}}^\a=S_\alpha\cap\s_{\norm{x}}$,
$X_\t^\a=S_\alpha\cap X_\t$ and $N_x^\a=S_\a\cap
N(\e,\abs{\tilde{f}(x)})$. Let $\pi_{\alpha_x}$ be the orthogonal
projection of $\C^n$ onto $T_xS_\alpha$, then the vector
$\pi_{\alpha_x}(i\grad\log \tilde{f}(x))$ is \emph{normal} to
$X_\t^\a$, the vector $\pi_{\alpha_x}(\grad\log \tilde{f}(x))$ is
\emph{normal} to the Milnor tube $N_x^\a$, and the vector
$\pi_{\alpha_x}(x)$ is \emph{normal} to $\s_{\norm{x}}^\a$. Denote
by $\spn_\R\bigl(\pi_{\alpha_x}(x),\pi_{\alpha_x}(i\grad\log
\tilde{f}(x))\bigr)$ the \emph{real} plane spanned by the vectors
$\pi_{\alpha_x}(x)$ and $\pi_{\alpha_x}(i\grad\log \tilde{f}(x))$.
By \eqref{eq:fib.esph} we have that
\begin{equation*}
 (T_x\mathfrak{F}_x^\a)^\perp=\spn_\R\bigl(\pi_{\alpha_x}(x),\pi_{\alpha_x}(i\grad\log
 \tilde{f}(x))\bigr),
\end{equation*}
the orthogonal space is taken in the tangent space to the stratum.

As in \cite[Lem.~4.6]{Mi2} we have to consider two cases:
\vspace*{12pt}

\noindent\textbf{Case 1:} The vectors $\pi_{\alpha_x}(x)$ and $\pi_{\alpha_x}(\grad\log \tilde{f}(x))$
are linearly independent over $\C$.
\vspace*{12pt}

We have that $\s_{\norm{x}}^\a$ and $N_x^\a$ are codimension 1 submanifolds of $S_\a$ and
the vectors $\pi_{\alpha_x}(x)$ and $\pi_{\alpha_x}(\grad\log \tilde{f}(x))$ are respectively normal to them.
Since these vectors are linearly independent over $\R$, we have that $\s_{\norm{x}}^\a$ and $N_x^\a$ are transverse.
Let $M_x^\a$ be their intersection, which is a submanifold of $\s_{\norm{x}}^\a$ of codimension 1. The fibre
$\mathfrak{F}_x^\a$ is also a codimension 1 submanifold of $\s_{\norm{x}}^\a$.

\medskip

\noindent \textbf{Claim:} The manifolds $M_x^\a$ and $\mathfrak{F}_x^\a$ are transverse in $\s_{\norm{x}}^\a$.
\begin{proof}[Proof of Claim]
Suppose they are not transverse, then $T_x M_x^\a=T_x\mathfrak{F}_x^\a$, hence
$(T_x M_x^\a)^\perp=(T_x\mathfrak{F}_x^\a)^\perp$. But
\begin{align*}
 (T_x M_x^\a)^\perp&=\spn_\R\bigl(\pi_{\alpha_x}(x),\pi_{\alpha_x}(\grad\log \tilde{f}(x))\bigr),\\
(T_x\mathfrak{F}_x^\a)^\perp&=\spn_\R\bigl(\pi_{\alpha_x}(x),\pi_{\alpha_x}(i\grad\log
\tilde{f}(x))\bigr),
\end{align*}
here all the orthogonal spaces are considered in  the tangent
space to the stratum $T_xS_\a$. Therefore the vector
$\pi_{\alpha_x}(x)$ is in the \emph{real} plane generated by
$\pi_{\alpha_x}(\grad\log \tilde{f}(x))$ and
$\pi_{\alpha_x}(i\grad\log \tilde{f}(x))$, this implies that
$\pi_{\alpha_x}(x)$ is in the \emph{complex} line generated by
$\pi_{\alpha_x}(\grad\log \tilde{f}(x))$, which contradicts our
original assumption.\hfill$\blacksquare$
\renewcommand{\qed}{}
\end{proof}

Since $M_x^\a$ and $\mathfrak{F}_x^\a$ are transverse in
$\s_{\norm{x}}^\a$ there is a \emph{unique} direction in
$T_xM_x^\a$ which is not in $T_x\mathfrak{F}_x^\a$ and is
orthogonal to $T_x(M_x^\a\cap\mathfrak{F}_x^\a)$ . Let $\nu(x)$ be
a unit vector in this direction, then the \emph{real} plane
$H_x=\spn_\R\bigl(\nu(x),\pi_{\alpha_x}(x)\bigr)$ is complementary
to $T_x\mathfrak{F}_x^\a$ in $T_xS_\a$, that is
\begin{equation*}
 T_xS_\a=T_x\mathfrak{F}_x^\a\oplus H_x.
\end{equation*}

Let $W$ be a neighbourhood of $x$ where there is a rugose vector
bundle $F$ on $W\cap X$ tangent to the stratification $\{S_\a\cap
W\}$ which induces the tangent bundle $TS_\a$ on the stratum
$S_\a\cap W$ \cite[Cor.~(4.5)]{Ver}. Consider the rugose vector
bundle morphism $T\mathfrak{F}\colon F\to \mathfrak{F}^*T\C$
induced by $\mathfrak{F}$ and let $K$ be its kernel. For a point
$x\in S_\a$ we have that $F_x=T_xS_\a$ and the fibre
$K_x=T_x\mathfrak{F}_x^\a$.

Let $x'\in W$ and let $\pi_{F_{x'}}$ be the orthogonal projection
of $\C^n$ onto the fibre $F_{x'}$ of $F$ at $x'$. Since the
vectors $\pi_{\alpha_x}(x)$ and $\pi_{\alpha_x}(\grad\log
\tilde{f}(x))$ are linearly independent over $\C$, the vectors
$x'$ and $\grad\log \tilde{f}(x')$ are also linearly independent
over $\C$ for every $x'$ in a neighbourhood of $x$ which we can
assume is also $W$. Then, by the previous arguments, there exists
a unit vector $\tilde{\nu}(x')$ tangent to the intersection
$M_{x'}$ of the sphere   $\s_{\norm{x'}}$ and the Milnor tube
$N(\e,\abs{\tilde{f}(x')})$ and orthogonal to $M_{x'}\cap
\mathfrak{F}_{x'}$ where $\mathfrak{F}_{x'}$ is the fibre of
$\mathfrak{F}$ which contains $x'$. Define
$H_{x'}=\pi_{F_{x'}}\Bigl(\spn_\R\bigl(\tilde{\nu}(x'),x'\bigr)\Bigr)$.
Notice that for $x\in S_\alpha$ we have that
$\nu(x)=\pi_{\alpha_x}(\tilde{\nu}(x))$ and
$\pi_{F_x}=\pi_{\alpha_x}$, so both definitions of $H_x$ coincide.
Since $H_x$ is complementary to $K_x$ in $F_x$, $H_{x'}$ is also
complementary to $K_{x'}$ in $F_{x'}$ for every $x'$ in a
neighbourhood of $x$ which we can also assume is $W$. Hence $H$ is
a horizontal subbundle of $F$ and it gives a rugose stratified
lifting $w$ of $\bar{w}$ in $W$.

Let $p(t)$ be an integral curve of the vector field $w$.
Since $d\mathfrak{F}_x(w(x))=\bar{w}(\mathfrak{F}(x))$, we have that
\begin{equation*}
 \theta(p(t))=t+\text{constant},
\end{equation*}
and by \eqref{eq:d.theta},
\begin{align*}
  1=\frac{d\theta(p(t))}{dt}&=\Re\inpr{w(p(t))}{i\grad\log \tilde{f}(p(t))}\\
&=\Re\inpr{w(p(t))}{\pi_{F_{p(t)}}(i\grad\log \tilde{f}(p(t)))}.
 \end{align*}
In particular we have that
\begin{equation}\label{eq:real.part}
\Re\inpr{w(x)}{i\grad\log \tilde{f}(x)}=\Re\inpr{w(x)}{\pi_{F_x}(i\grad\log \tilde{f}(x))}=1.
\end{equation}
Notice that by definition, $w(x)\in T_x\s_{\norm{x}}^\a\cap H_x$ and therefore it is a \emph{real} multiple of
the vector $\nu(x)$.
Since $\nu(x)$ is tangent to the Milnor tube $M_x^\a$ and $\pi_{\alpha_x}(\grad\log \tilde{f}(x))$
is normal to $M_x^\a$, we have that
\begin{equation}\label{eq:imaginary.part}
\Re\inpr{w(x)}{\grad\log \tilde{f}(x)}=\Re\inpr{w(x)}{\pi_{F_x}(\grad\log \tilde{f}(x))}=0.
\end{equation}
Combining \eqref{eq:real.part} and \eqref{eq:imaginary.part} we have that
\begin{equation*}
 \inpr{w(x)}{i\grad\log \tilde{f}(x)}=1,
\end{equation*}
which implies that
\begin{equation}
 \abs{\arg \inpr{w(x)}{i\grad\log \tilde{f}(x)}}<\frac{\pi}{4}.
\end{equation}
This condition certainly holds throughout a neighbourhood of $x$.

\vspace*{12pt}
\noindent\textbf{Case 2:} The vector $\pi_{\alpha_x}(\grad\log \tilde{f}(x))$ is equal to a multiple $\lambda\pi_{\alpha_x}(x)$.
\vspace*{12pt}

In this case, the \emph{complex} line generated by
$\pi_{\alpha_x}(\grad\log \tilde{f}(x))$ is the \emph{real} plane
$H_x=\spn_\R\bigl(\pi_{\alpha_x}(x),\pi_{\alpha_x}(i\grad\log
\tilde{f}(x))\bigr)=\spn_\R\bigl(\pi_{\alpha_x}(ix),\pi_{\alpha_x}(x)\bigr)$.
Since $H_x$ is the orthogonal space in $T_xS_\a$ to
$T\mathfrak{F}_x^\a$, $H_x$ is complementary to
$T\mathfrak{F}_x^\a$ in $T_xS_\a$.

As before, let $W$ be a neighbourhood of $x$ where there is a
rugose vector bundle $F$ on $W\cap X$ tangent to the
stratification $\{S_\a\cap W\}$ which induces the tangent bundle
$TS_\a$ on the stratum $S_\a\cap W$. Consider the rugose vector
bundle morphism $T\mathfrak{F}\colon F\to \mathfrak{F}^*T\C$
induced by $\mathfrak{F}$ and let $K$ be its kernel. For a point
$x\in S_\a$ we have that $F_x=T_xS_\a$ and the fibre
$K_x=T_x\mathfrak{F}_x^\a$.

Let $x'\in W$ and let $\pi_{F_{x'}}$ be the orthogonal projection of $\C^n$ onto the fibre $F_{x'}$ of $F$ at $x'$.
Define $H_{x'}=\pi_{F_{x'}}\bigl(\spn_\R(ix',x')\bigr)$.
For $x\in S_\alpha$ we have that $\pi_{F_x}=\pi_{\alpha_x}$, so both definitions of $H_x$ coincide.
Since $H_x$ is complementary to $K_x$ in $F_x$, $H_{x'}$ is also complementary to $K_{x'}$
in $F_{x'}$ for every $x'$ in a neighbourhood of $x$ which we can also assume is $W$. Hence $H$ is a horizontal
subbundle of $F$ and it gives a rugose stratified lifting $w$ of $\bar{w}$ in $W$.

Notice that by definition when $x \in S_\a$, $w(x)\in
T_x\s_{\norm{x}}^\a\cap H_x$, so there exists $k\in \R$ such that
$w(x) = k \pi_\a(ix)$. Therefore we have:
\begin{align*}
 \inpr{w(x)}{i\grad\log \tilde{f}(x)}&=\inpr{w(x)}{\pi_{\alpha_x}(i\grad\log \tilde{f}(x))}\\
&=k \inpr{\pi_{\alpha_x}(ix)}{\pi_{\alpha_x}(i\grad\log \tilde{f}(x))}\\
&=k \inpr{\pi_{\alpha_x}(ix)}{\lambda\pi_{\alpha_x}(ix)}\\
&=k \bar{\lambda}\norm{\pi_{\alpha_x}(x)}^2.
\end{align*}
Since $d\mathfrak{F}_x(w(x))=\bar{w}(\mathfrak{F}(x))$, we have that
\begin{equation}\label{eq:imagin.c2}
 \Re\inpr{w(x)}{i\grad\log \tilde{f}(x)}=\Re\inpr{w(x)}{\pi_{\a_x}(i\grad\log \tilde{f}(x))}=1,
\end{equation}
and by Lemma~\ref{prop:arg.lambda} and since $k$ is real
\begin{equation*}
 \abs{\arg \inpr{w(x)}{i\grad\log \tilde{f}(x)}}<\frac{\pi}{4}.
\end{equation*}
Again, this condition holds throughout a neighbourhood of $x$ and using a rugose partition of unity
we obtain a global vector field $w$ which lifts $\bar{w}$ and satisfies the following properties:
\begin{align*}
 \Re\inpr{w(x)}{i\grad\log \tilde{f}(x)}&=1,\\
\abs{\arg \inpr{w(x)}{i\grad\log \tilde{f}(x)}}&<\frac{\pi}{4}.
\end{align*}
Therefore
\begin{equation*}
 \abs{\Re\inpr{w(x)}{\grad\log \tilde{f}(x)}}=\abs{\Im\inpr{w(x)}{i\grad\log \tilde{f}(x)}}<1 \,.
\end{equation*}

To guarantee that $p(t)$ cannot tend to $V$ as $t$ tends to some finite limit $t_0$ is equivalent
to insure that $f(p(t))$ cannot tend to zero, or that $\log\abs{f(p(t))}$ cannot tend to $-\infty$,
as $t\to t_0$. But we have that
\begin{equation*}
\Biggl|\frac{d\log \abs{f(p(t))}}{dt}\Biggr|=\Biggl|\frac{d\Re\log f(p(t))}{dt}\Biggr|=
\abs{\Re\inpr{w(p(t))}{\grad\log \tilde{f}(p(t))}}<1 \,.
\end{equation*}
Hence $\log\abs{f(p(t))}<t+\text{constant}$ and therefore $\abs{f(p(t))}$ is bounded away from zero
as $t$ tends to any finite limit.

Taking $\e_0>\e>0$ we
have that $w$ satisfies the lemma on $(X\cap\B_\e)\setminus V$.
\end{proof}

We notice that Proposition \ref{lem:vec.fld.sph} essentially
proves that the map $$\Psi: (X \cap \B_\e) \setminus V  \to
\RP^1$$ in Theorem \ref{Thm:fib.thm}, is the projection map of a
fibre bundle (see Section \ref{Section-Proof of Theorem} below).

\medskip

\begin{proof}[Proof of Theorem~\ref{thm.spherefication}] That the restriction of $\mathfrak{F}$
to each sphere around $\0$ of radius $\e' \le \e$ is a fibre
bundle over the corresponding circle of radius $\e'$, is an
immediate consequence of Proposition \ref{lem:vec.fld.sph}. The
composition of this restriction with the radial projection of
$\s^1_{\e'}$ onto $\s^1$ is the Milnor fibration $\phi$ in
\eqref{M-fib}. To complete the proof we use the uniform conical
structure given in Theorem~\ref{Thm:Can.Dec}-{\bf\ref{it:cones})}
which gives the local triviality of $\mathfrak{F}$ over
$\C\setminus\{0\}$.
\end{proof}

\subsection{Proof of Theorem~\ref{Thm:fib.thm}}\label{Section-Proof of Theorem}

 Now consider the maps of Theorem~\ref{Thm:fib.thm}: $\Psi(x)=
(Re(f(x)) : Im(f(x))),$ and  $ \Phi(x)= \frac{f(x)}{\abs{f(x)}}.$
Notice that $\Phi$ is a lifting of $\Psi$ to the double cover
$\s^1$ of $\RP^1$, so we have the following commutative diagram
\begin{equation*}
\xymatrix{
(X\cap\B_\e)\setminus V \ar[r]^-{\Phi}\ar[rd]_{\Psi}&\s^1 \ar[d]\\
 &\RP^1
}
\end{equation*}
From Remark~\ref{rem:gluing.links}, for each $ \mathcal{L}_\t\in\RP^1$ and each $\t \in [0,2\pi)$ one
has,
\begin{equation*}
\Psi^{-1}(\mathcal{L}_\t) = E_\t\cup E_{\t+\pi} \quad \hbox{and}
\quad \Phi^{-1}(e^{i\t}) = E_\t \;.
\end{equation*}

The vector field $w$ constructed in
Proposition~\ref{lem:vec.fld.sph} provides topological
trivialisations around the fibres of the maps $\Psi$ and $\Phi$,
showing that both are fibre bundles, which proves the first
statement of Theorem~\ref{Thm:fib.thm}. That the restriction of
$\Phi$ to the link of $X$ is the classical Milnor fibration is
immediate from Theorem~\ref{thm.spherefication}.

\begin{remark}
 If $X\setminus V$ is non-singular the flow obtained in Proposition~\ref{lem:vec.fld.sph} can be made $C^\infty$,
thus all the $E_\t$ are diffeomorphic.
\end{remark}

On the other hand, let $\D_\eta$ be a disc in $\C$ of radius
$\eta$ where $\e
>> \eta
>0$ and consider the Milnor tube $N(\e,\eta) = X \cap\B_\e  \cap
f^{-1}(\partial \D_\eta)$. Since the restriction of $f$ to $(X\cap
\B_\e)\setminus V$ is a submersion on each stratum, we have that
$\set{N(\e,\eta)\cap S_\a}{S_\a\subset X\setminus V}$ is a
Whitney-strong stratification of $N(\e,\eta)$
\cite[Rem.~(3.7)]{Ver}. Since the stratification $\{S_\a\}$ of $X$
satisfies Thom's $(a_f)$ condition, if $\eta$ is small enough,
then all the fibres of $f$ in $N(\e,\eta)$ are transverse to
$\s_\e$. Using Thom-Mather first Isotopy lemma and the
transversality of the fibres with the boundary as in \cite[\S
1]{Le1}, we obtain that the restriction of $\Phi$ to $N(\e,\eta)$,
\begin{equation*}
\bar{\Phi}\colon N(\e,\eta)\to \s^1  \, ,
\end{equation*}
is also a fibre bundle. This map equals the restriction of $f$ to
$N(\e,\eta)$ followed by the radial projection of
$\partial\D_\eta$ onto $\s^1$, so this is the Milnor-L\^e
fibration \eqref{ML-fib} up to multiplication by a constant.

\medskip It remains to prove that the two fibrations \eqref{M-fib} and \eqref{ML-fib}
are equivalent. We need the following, which is a  consequence of
  Lemma~\ref{lem:g.sub}.

\begin{proposition}\label{lem:uni.con.str.0}
Let $\{S_\alpha\}$ be a Whitney stratification of $X$ adapted to
$V$  and for each $\t \in [0,\pi)$ equip  $X_\t$ with
the stratification $\{X_\t\cap S_\a\}$  obtained by
intersecting $X_\t$ with the strata of $\{S_\alpha\}$. Then for every
sufficiently small ball $\B_\e$ around $\0$,
 there exists  a stratified, rugose vector field $\tilde{v}$ on
 $(X\cap\B_\e) \setminus V$ which has the following properties:
\begin{enumerate}[i)]\setlength{\itemsep}{0pt}
\item
It is radial, {\it i.e.}, it is transverse to the
intersection of $X$ with all spheres in $\B_\e$ centred at
$\0$.\label{it:pr1.l}

\item It is tangent to the strata of each $X_\t
\setminus V$\label{it:pr2.l}.

\item
It is transverse to all the tubes $f^{-1}(\partial
\D_\eta)$.\label{it:pr3.l}
\end{enumerate}
\end{proposition}

\begin{proof}
Notice that the vector field $v_{\mathfrak{F}}$ in the proof of
Lemma~\ref{lem:uni.con.str} already satisfies conditions
\ref{it:pr1.l}) and \ref{it:pr2.l}) and is rugose. The aim now is
to modify this vector field to insure that it satisfies also
\ref{it:pr3.l}), being rugose.

Let $\e_0$ be as in Lemma~\ref{it:St.trans.Sp}. Let $\e_0>\e>0$ be
small enough so that the restrictions of $f$ and $\mathfrak{F}$
 to $(X\cap
\B_\e)\setminus V$
are both submersions on each stratum; such an $\e$ exists by the
theorem of Bertini-Sard-Verdier \cite[Thm.~(3.3)]{Ver}. Let $u$ be the radial vector field in the proof of
Theorem~\ref{thm.spherefication}. Now we
use $f$ to lift $u$ to a stratified
rugose vector field $v_f$  on $(X\cap \B)\setminus V$
such that, for every $x\in(X\cap \B)\setminus V$, we have:
\begin{equation*}
df_x(v_f(x))=u(f(x))  \;.
\end{equation*}
The local flow associated to $v_f$ is transverse to all Milnor
tubes $f^{-1}(\partial \D_\eta)$, while the one associated to
$v_{\mathfrak{F}}$ is transverse to all spheres in $\B_\e$ centred at
$\0$.
The integral paths of both move along $X_\t\setminus V$,
{\it i.e.}, along points where the argument $\t$ of $f$ does not
change.

From the definitions of $v_f$ and $v_{\mathfrak{F}}$ and the fact
that each stratum $S_\a\cap X_\t$ of $X_\t$ is transverse to all
the spheres (Lemma~\ref{it:St.trans.Sp}), one can see that the
vectors $v_f(x)$ and $v_{\mathfrak{F}}(x)$ cannot point in
opposite directions for any $x\in (X\cap \B_\e)\setminus V$. Hence
adding up $v_f$ and $v_{\mathfrak{F}}$ on $(X\cap\B_\e)\setminus
V$ we get a vector field $\tilde{v}$ which satisfies the three
properties of Proposition~\ref{lem:uni.con.str.0}.
\end{proof}


Now let $\phi$ denote the restriction of $\Phi$ to $L_X \setminus
L_f = \s_\e \cap (X \setminus V)$, which defines the classical
Milnor fibration. The flow associated to the  vector field $\tilde
{v}$ in  Proposition~\ref{lem:uni.con.str.0}  defines in the usual
way a homeomorphism between the fibre of ${f}^{-1}(e^{i\t})  \cap
\B_\e$ and the portion of the fibre $\phi^{-1}(e^{i\t})$ defined
by the inequality $\abs{f(x)}\ge \eta$.

To complete the proof we must show that the fibration defined by
$\phi$ on $L_X \setminus L_f$ is equivalent to the restriction of
$\phi$ to the points in the sphere satisfying $\abs{f(x)}>\eta$.
For this we use that since $f$ satisfies Thom's $(a_f)$ condition,
the restriction of $f$ to $T(\e,\eta) = \s_\e \cap f^{-1}(\D_\eta
\setminus \{0 \})$ is a submersion on each stratum. Hence, again
by Verdier's result Proposition~\ref{rugose lifting}, we can lift
the radial vector $u(z) = z$ on $\D_\eta \setminus \{0 \}$ to a
stratified, rugose vector field on $T(\e,\eta)$, whose flow
preserves the fibres of $\phi$ and is transverse to the
intersection with $\s_\e$ of all the Milnor tubes $f^{-1}(\partial
\D_{\eta'})$ for all $0 < \eta' \le \eta$. This gives the
equivalence of the two fibrations, and therefore finishes the
proof of Theorem~\ref{Thm:fib.thm}.

\medskip

\section{A fibration theorem on the blow up}\label{section-Blow-up}

In this section  we complete the proof of
statement~{\bf\ref{it:decomposition})} of
Theorem~\ref{Thm:Can.Dec} by proving that the varieties $X_\t$ are
all homeomorphic. In order to do that, we realise the spaces
$X_\t$ as fibres of a topological fibre bundle.

A minimal way to obtain an unfolding of the pencil $(X_\t)$ is
achieved by the blow-up of its axis $V= f^{-1}(0)$. We produce in
this way a new analytic space ${\tilde X}$ with a projection to
$\R\P^1$ whose fibres are exactly the $X_\t$'s. The space ${\tilde
X}$ is equipped with a Whitney stratification obtained in a
canonical way from the one on X. Using the Thom-Mather First Isotopy
Lemma, we prove that ${\tilde X}$ is a fibre bundle over $\R\P^1$. This
is a new Milnor-type fibration theorem in which we do not need
any more to remove the zero locus of the function $f$.


\subsection{The real blow up}
As before, we consider a sufficiently small representative $X$ of
the complex analytic germ $(X,\0)$  and $f\colon (X,\0)
\rightarrow (\C,0)$ holomorphic. Set $V= f^{-1}(0)$ and consider
the real analytic  map
$$\begin{array}{rrcl}
 \Psi \colon & X \setminus V & \rightarrow & \R\P^1 \\
 & z & \mapsto & (Re(f(z)) :  Im(f(z)))\,.\\
 \end{array}$$

Let $\tilde{X}$ be the analytic set in
 $X \times \R\P^1$ defined by  $Re(f)t_2 - Im(f)t_1
 =0$, where $(t_1 : t_2)$ is a system of homogeneous coordinates
 in $\R\P^1$. The first projection induces a real analytic map:
 $$e_V: \tilde{X} \rightarrow X\,;$$
 this is the real blow-up of $V$ in $X$ \cite[\S 3]{Mather:TopStab}.
 It induces a real analytic
 isomorphism  $\tilde{X} \setminus e_V^{-1}(V) \cong X \setminus V$. The
 inverse image of $V$ by $e_V$ is $V \times \R\P^1$.

 The second projection induces a real analytic map:
 $$\tilde{\Psi}: \tilde{X} \rightarrow \R\P^1\,,$$
and one has: $\tilde{\Psi}\mid _{\tilde{X}\setminus e_V^{-1}(V)} =
\Psi \circ e_V\mid _{\tilde{X}\setminus e_V^{-1}(V)}$, {\it i.e.},
 $\tilde{\Psi}$ extends $\Psi$ to $e_V^{-1}(V) \cong V\times
\R\P^1$.
 It is clear that each fibre $\tilde{\Psi}^{-1}(t)$ is isomorphic
to $X_{\theta}\times \{t\}$, where $\t$ is the angle between the
horizontal axis and the line represented by $t$.

\subsection{The fibration theorem on the blow up}
We now prove Theorem~\ref{Thm:blow.up}, {\it i.e.}, that
$$\tilde{\Psi}: \tilde{X} \rightarrow \R\P^1\,,$$ is a
topological fibre bundle with fibres the $X_\t$.

For this, we consider a Whitney stratification $\{S_{\alpha}\}$ of
$X$ adapted to $V$. This induces a
stratification ${\Sigma_{\alpha}}$ on $\tilde{X}$ defined by:

$$\Sigma_{\alpha} = e_V^{-1}(S_{\alpha}).$$
We claim this  stratification is Whitney regular. This is
Proposition~\ref{stratification-blow-up} at the end of the
section, and we assume it for the moment.

Notice that the map $\tilde{\Psi}$ restricted to each stratum
$\Sigma _\alpha$ has no critical point. In fact, outside
$e_V^{-1}(V)$, $\tilde{\Psi}$ coincides with $\Psi$ (up to
isomorphism), which is a submersion onto $\R\P^1$  restricted to
each stratum of $X\setminus V$. On the other hand, the restriction
of $\tilde{\Psi}$ to a stratum contained in $e_V^{-1}(V)$ is
surjective onto $\R\P^1$ and is given by the second projection on
the product $V\times \R\P^1$.

In order to be able to apply Thom-Mather's First Isotopy Lemma we
need to restrict the map  $\tilde{\Psi}$ to a compact subset where
it is a proper submersion (see
\cite[Proposition~11.1]{Mather:TopStab}).

Consider the closed ball $\bar{\B}_\varepsilon$ centred at $\0$
with radius $\varepsilon$, in the ambient space $\C^n$. The
intersection ${\tilde X}$ with $\bar{\B}_\varepsilon\times \R\P^1$
is compact and hence the restricted map $\tilde{\Psi}: {\tilde
X}\cap (\bar{\B}_\varepsilon\times \R\P^1) \rightarrow \R\P^1$ is
proper. This map is a submersion in the stratified sense if and
only if the fibres of $\tilde{\Psi}$ are transverse to
$\s_\varepsilon\times \R\P^1$ or, equivalently, the spaces
$X_\theta$ are all transverse to $\s_\varepsilon$ for sufficiently
small $\varepsilon$. But this is already given by
Lemma~\ref{it:St.trans.Sp}.

Theorem~\ref{Thm:blow.up} obviously implies that the spaces $X_\t$ are all
homeomorphic.

\begin{remark}\label{Rk:iso-sing}
The usual proof of the Thom-Mather First Isotopy Lemma  uses that
the map  $\tilde{\Psi}$   is a proper submersion on each stratum,
in order to lift vector fields from the target space, in this case
$\R\P^1$. In the case when both $X$ and $f$ have an isolated
singularity at $\0$, then each $X_\t$ also has an isolated
singularity at $\0$ and  the canonical vector field on $\R P^1
\cong \mathbb S^1$ can be lifted to a rugose vector field
$\widetilde v$ on the blow-up $\tilde{X}$, which is $C^\infty$
away from the stratum $\{0\} \times \R P^1$ and leaves invariant
the subspace $V \times \R P^1$. This will be used later, for
proving Corollary ~\ref{Cor:Mil.fib}.
\end{remark}

 \medskip

\begin{remark}
Theorem~\ref{Thm:blow.up} gives another proof that the stratification induced on the $X_\t$ is Whitney regular.
Since Whitney regularity conditions are stable under
transversality (see for example \cite[Chap.~1 (1.4)]{GWPL}), the fibres
of the projection ${\tilde \Psi}: {\tilde X} \rightarrow \R\P^1$
inherit a Whitney regular stratification given by $\{\Sigma_\a
\cap {\tilde \Psi}^{-1}(t)\}$. Hence $\{S_\a\cap X_\t\}$ is a
Whitney regular stratification of each $X_\t$.
\end{remark}

\subsection{An induced Whitney stratification on the blow up}

The proposition below completes  the proofs of Theorems~\ref{Thm:Can.Dec}--{\bf\ref{it:decomposition})}
 and \ref{Thm:blow.up}.

\begin{proposition}\label{stratification-blow-up}

Consider a Whitney stratification $\{S_{\alpha}\}$ of a
sufficiently small representative $X$ of the germ $(X,\0)$ adapted
to $V$. Then the partition $\{\Sigma_{\alpha}\}$ given by the
inverse images by $e_V$ of the strata $S_{\alpha}$ is a Whitney
stratification of ${\tilde X}$.
\end{proposition}

\begin{proof}
 Note that if $S_{\alpha}$ is a stratum contained in $V$, the
corresponding stratum $\Sigma_{\alpha} \subset {\tilde X}$ is
 $S_{\alpha}\times \R\P^1$.

We first  check that this partition of ${\tilde X}$ satisfies the
frontier condition. Let $S_{\alpha}$ and $S_{\beta}$ be strata
contained respectively in $V$ and $\sing(X)$, such that
$\Sigma_{\alpha}\cap {\overline {\Sigma_{\beta}}} \neq \emptyset$.
The corresponding strata in $X$ satisfy the inclusion $S_\a
\subset {\overline {S_\beta}}$.

Let $(y,t_1:t_2)$ be a point of $\Sigma_{\alpha}$, {\it i.e.}
$f(y)=0$ and $y\in {\overline {S_\beta}}$. Since $S_\beta$ is not
contained in $V$, the restriction of $f$ to a neighbourhood of $y$
in ${\overline {S_\beta}}$ is surjective over a neighbourhood of
the origin in $\C$. Hence for every line ${\mathcal L}_\theta
\subset \R^2$ containing the origin, the intersection $X_\theta
\cap S_\beta$ has points $y_n$ arbitrarily close to $y$. Taking
the line corresponding to the projective point $(t_1 : t_2)$, we
obtain a sequence $(y_n, t_1 : t_2) \in \Sigma_\beta$ converging
to $(y, t_1 : t_2)$.

In the other cases, it is clear that the frontier condition holds.

Now, we prove that this stratification is Whitney regular.
Consider a stratum $\Sigma_\alpha = S_\alpha \times \R\P^1\subset
V\times \R\P^1$ and a stratum $\Sigma_\beta$ not contained in
$V\times \R\P^1$ such that $\Sigma_\alpha \subset {\bar
\Sigma}_\beta$. For the other cases it is easy to see that Whitney
conditions hold.

Consider a sequence of points $(x_n, t_n)\in \Sigma_\beta \subset
X\times \R\P^1$ converging to a point $(x,t)\in S_\alpha \times
\R\P^1$ such that the sequence of tangent spaces $T_n:=
T_{(x_n,t_n)}\Sigma_\beta$ converges to a linear space $T$.

\begin{lemma}\label{ab-bnotzero}
For any $b\in \R$ there exists a vector $a\in \C^n$ such that the
vector $(a,b)\in T$.
\end{lemma}

\begin{proof}

Consider the following commutative diagram:
\begin{equation*}
 \xymatrix{
{\tilde X}\ar[r]^{e_V}\ar[d]_{\pi} & X\ar[d]^{\psi}\\
*+[l]{U\times \R\P^1 \supset {\tilde U}}\ar[r]^{e_0} & *+[r]{U \subset \R^2}
}
\end{equation*}
where $\psi$ is the real analytic map on $X$ defined by
$(Re(f), Im(f))$, $e_0: {\tilde U} \rightarrow U$ is the blow-up
of the origin in a neighbourhood $U$
 of $0$ in $\R^2$ and $\pi$ is the pull-back of $\psi$ by  $e_0$.

Let us work on a local chart $X\times \R$ of $X\times \R\P^1$ near
the point $(x,t)$.
 Call $(y_n, t_n)\in {\tilde U}$ and $(0,t)\in \{0\}\times \R$
 the respective images of $(x_n,t_n)$ and $(x,t)$ by $\pi$.
 The direction of the tangent
 space $T_{(0,t)}(\{0\}\times \R)= \{0\} \times \R$ is naturally
 contained in the limit of directions of tangent spaces ${\mathcal
 T}_n:=T_{(y_n,t_n)}{\tilde U}$.
 Hence, there exists a sequence of directions of lines $l_n \subset
 {\mathcal T}_n$ converging to $\{0\}\times \R$.

 Since the map $\pi$ is a submersion at $(x_n,t_n)$, there exists a
 sequence of directions of lines $h_n\subset T_n$ whose images by
 the tangent map to $\pi$ at $(x_n,t_n)$ is $l_n$. Since $l_n$
 converges to $\{0\}\times \R$, the limit $h$ of $h_n$ is not
 contained in $\C^n\times \{0\}$. So there exists a vector
 $(a,b)\in T$ such that that $b\neq 0$.

\end{proof}

\begin{lemma}\label{condition-a}
The stratification $\{\Sigma_\alpha\}$ satisfies Whitney's
$(a)$-condition.
\end{lemma}

\begin{proof}

Under the previous notations, we need to prove that the direction
of tangent space $T_{(x,t)}\Sigma_\alpha$ is contained in $T$.

Let us now describe the tangent space $T_n$. Recall that the
blow-up  $\tilde X$ is defined as the subspace of $X\times \R\P^1$
given by the equation $t_2Re(f)- t_1Im(f) =0$, where $(t_1:t_2)$
are homogeneous coordinates in $\R\P^1$.

Suppose the point $(x,t)\in S_\alpha \times \R\P^1$ lives in the
local chart given by $t_1 \neq 0$ or, equivalently, $Re(f)\neq 0$,
and consider a holomorphic extension ${\tilde f}$   of $f$ to the
ambient space $\C^n$. Denote by $f_1$ and $f_2$ respectively the
real and imaginary parts of $\tilde f$.

We then have:
$$T_n= \{(a,b) \in T_{x_n}S_\beta\times \R,  \langle (\frac{t_2}{t_1} \grad_{x_n} f_1 -
\grad_{x_n} f_2) \,,\, a \rangle + f_1(x_n)b =0\}$$

\noindent where $\langle\ , \ \rangle $ is the real scalar (or
inner) product in a real vector space.

If we call $u_n$ the vector $((\frac{t_2}{t_1} \grad_{x_n} f_1 -
\grad_{x_n} f_2), f_1(x_n))\in \C^n\times \R$, $L_n$ the line
generated by $u_n$ and $N(L_n)$ its orthogonal linear space in
$\C^n\times \R$, then we can write:

$$T_n = (T_{x_n}S_\beta\times \R) \cap N(L_n).$$

We  now prove that the sequence of lines $L_n$ converges to a line
$L$ contained in $\C^n\times \{0\}$. This is a consequence of the
{\L}ojasiewicz inequality \cite[p.~92]{Loj}.

In fact, if $g$ is a real analytic map in a neighbourhood of a
point $p\in \R^n$ then there exist a neighbourhood $p\in W\subset
\R^n$ and $0<\theta < 1$ such that, for any $q\in W$ we have:
\begin{equation}\label{lojasiewicz}
\norm{f(q) - f(p)}^{\theta} \leq \norm {\grad_q f}.
\end{equation}

We will apply this inequality to $f_1$ in a neighbourhood of $x$.

Since $\tilde f$ is holomorphic, the vectors $\grad f_1$ and
$\grad f_2$ are orthogonal and have the same module at any point,
so:

$$\norm {\frac{t_2}{t_1} \grad_{x_n} f_1 - \grad_{x_n}
f_2}^2 = \norm {\grad_{x_n}f_1}^2 ((\frac{t_2}{t_1})^2 +1).$$

The function ${\tilde f}$ being holomorphic, it has an isolated
critical value at $0$ and hence $\grad_{x_n} f_1 \neq 0$. So
dividing the vector $u_n$ by the module of $\grad_{x_n}f_1$ we
have that
$$\frac{\norm{\frac{t_2}{t_1} \grad_{x_n} f_1 -
\grad_{x_n} f_2} }{\norm {\grad_{x_n} f_1}} $$ tends to  a
 non-zero  value,
 while by the inequality \eqref{lojasiewicz} one has:
$$\frac{\norm{f_1(x_n)}}{\norm{\grad_{x_n} f_1}}< \norm
{f_1(x_n)}^{1-\theta}\,,$$ and hence it tends to zero. So the
sequence of lines $L_n$ tends to a line $L$ contained in
$\C^n\times \{0\}$ and  the normal spaces $N(L_n)$ converge to a
linear space containing $\{0\} \times \R.$

We are now going to prove that the limit $T$ of $T_n$ is equal to
the intersection of the limit of $T_{x_n}S_\beta \times \R$ with
the normal space $N(L)$ to $L$ in $\C^n \times \R$.

Since the linear space $N(L)$ is a hyperplane it is sufficient to
show that the limit of $T_{x_n}S_\beta \times \R$ is not contained
in $N(L)$, and then the limit of the intersection is the
intersection of the limits.

By Lemma \ref{ab-bnotzero}, for any $0\neq b \in \R$ there exists
a vector $a \in \C^n$ such that $(a,b)\in T$. Since $S_\beta$ has
real dimension at least two, then there exists $b\in \R$, $\neq 0$
and there exists $a\in \C^n$, $\neq 0$, such that $(a,b)\in T$.

So there exists a sequence $(a_n, b_n)\in T_n$ converging to
$(a,b)$. This means that $a_n \in T_{x_n}S_\beta$ and the scalar
product $\langle (a_n,b_n) \, , \, u_n\rangle =0$. If we write
$u_n = (u_{n,1},u_{n,2}) \in \C^n \times \R$, then
$$b_n = - \langle a_n \, , \, \frac{u_{n,1}}{u_{n,2}}\rangle.$$

Since $b_n$ converges to a non zero real value, the limit $a$ of
$a_n$ is not orthogonal in $\C^n$ to the limit $D$ of lines
generated by the vectors $\frac{u_{n,1}}{u_{n,2}}$. Notice that we
have the equality $L = D \times \{0\}$. So for any $s\in \R$ the
vector $(a,s)$ is not orthogonal to the line $L$, and then
$$\lim T_{x_n}S_\beta \times \R \nsubseteq N(L).$$
We conclude that
$$T= ((\lim T_{x_n}S_\beta) \times \R) \cap N(L).$$
On the other hand, the direction of tangent space to
$\Sigma_\alpha$ at $(x,t)$ is given by:
$$T_{(x,t)}\Sigma_\alpha = T_xS_\alpha \times T_t\R\P^1.$$
Since the stratification $\{S_\alpha\}$ on $X$ satisfies Whitney's
$(a)$-condition we have
$$T_xS_\alpha \subset \lim T_{x_n}S_\beta\,,$$
and hence we  obtain
$$T_{(x,t)}\Sigma_\alpha \subset \lim T_{(x_n,t_n)}\Sigma_\beta \,,$$
which proves the $(a)$-condition.
\end{proof}

We  now finish the proof of
Proposition~\ref{stratification-blow-up}. We use  condition $(a)$
to prove  condition $(b)$.

Keeping the previous notations, consider a sequence of points
$(y_n, s_n) \in \Sigma_\alpha = S_\alpha \times \R\P^1$ converging
to $(x,t)\in \Sigma_\alpha$ such that the sequence of lines $l_n$
joining the points $(x_n,t_n)$ and $(y_n, s_n)$ in $\C^n\times \R$
converges to a line $l$. We need to prove that $l\subset T$.

Consider the line $h_n$ generated in $\C^n$ by the vector
$x_n-y_n$. We can suppose the sequence of lines $h_n$ converges to
a line $h$. The lines $h_n$ and $h$ are respectively the
projection onto $\C^n$ of the lines $l_n$ and $l$. So if $(u,v)\in
\C^n \times \R$ is a directing vector of the line $l$, then $u$ is
a directing vector of $h$.

Since the strata $S_\beta$ and $S_\alpha$ satisfy Whitney's
$(b)$-condition, the line $h$ is contained in $\lim
T_{x_n}S_\beta$, and hence, there exists $v'\in \R$ such that the
vector $(u,v')\in \C^n\times \R$ is actually in the limit $T$ of
tangent spaces $T_n$.

By  Lemma \ref{condition-a}, the real line $\{0\}\times \R$ is
contained in T. So the vector $(u,v') + (0, v-v') \in T$,  and
 we obtain
$l\subset T,$ completing the proof of
Proposition~\ref{stratification-blow-up}.
\end{proof}


\subsection{Proof of Corollary~\ref{Cor:Mil.fib}}

Recall from Remark~\ref{rem:gluing.links} that we can write
$$X_{\t} = E_{\t} \cup V \cup E_{\t+\pi} \ .$$

\begin{lemma}\label{lem:acumulacion}
Assume $(X,\0)$ is irreducible and for every angle $\t$ let
$\overline E_\t$ denote the topological closure of $E_\t$. Then we
have:
$$\overline E_\t = E_\t \cup V.$$
\end{lemma}

\begin{proof}
Since $f$ is not constant in $(X,\0)$ and this germ is
irreducible, the subspace $V = f^{-1}(0)$ has complex codimension
1 in $(X,\0)$. So for each point $x \in V$ there exists a
neighbourhood $U\subset \C^n$, such that the restriction $f :
U\cap X \rightarrow \C$ is surjective onto an open neighbourhood
of $0$ in $\C$. Hence, for every line $\mathcal L$ through $\0$,
each  half line in ${\mathcal L} \setminus \{0\}$ intersects the image of the restriction of
$f$ in an open segment. In other words, for each $\t$,  $E_\t$
has a non-empty intersection with $X\cap U$.  Choosing the
neighbourhood $U$ arbitrarily small, we get that for each $\t$,
there exists a sequence of points in $E_\t$
converging to $x$. That is, every $x \in V$ is in $\overline
E_{\t}$. Since $E_\t\cup V$ is closed, we obtain the equality of the
lemma.
\end{proof}

We see from the preceding lemma that $E_\t$ and $E_{\t + \pi}$
have a common border on $V$. It is in this sense that we say that
they are glued together along $V$ forming the subspace $X_\t$.

\medskip

Notice that  the fibres of $\phi$ are precisely the intersections
of the sphere $\s_\e$ with the corresponding $E_\t$. So the fibres
of $\phi$ over two antipodal points of $\s^1$ are glued together
along the link $L_f$ forming the link $K_\t = X_\t \cap \s_\e$ of
$X_\t$. The fact that the link   $K_\t$ is homeomorphic to the
link of $\{Re \,f = 0\,\}$ is an immediate consequence of
Theorem~\ref{Thm:blow.up}. Furthermore,  if both $X$ and $f$ have
an isolated singularity at $\0$, then by Remark \ref{Rk:iso-sing},
all $X_\t \setminus \{\0\}$ are actually diffeomorphic.

It remains to prove the last statement in
Corollary~\ref{Cor:Mil.fib}, {\it i.e.}, that if both   $X$ and
$f$ have an isolated singularity at $\0$, then the link $K_\t$ of
each $X_\t$ is diffeomorphic to the double of the Milnor fibre of
$f$. This follows from the previous discussion and the {\it
folklore} theorem, that in this setting, Milnor's fibration
theorem gives an open-book decomposition of the sphere. In fact,
from the discussion above (see equation (\ref{eq:links}) in Remark
\ref{rem:gluing.links}) we have:
$$K_\t = (E_{\t} \cap
\s_\e) \, \cup \, (V \cap \s_\e) \, \cup \, (E_{\t+\pi} \cap
\s_\e) \, \, .$$ Then, by Remark \ref{Rk:iso-sing}, we have that
each $K_\t$   is a smooth real analytic, oriented manifold of
dimension $2 \hbox{dim}_\C X - 2$, obtained by gluing two Milnor
fibres of $f$ along their boundary by a smooth diffeomorphism,
given by a smooth flow. Such a diffeomorphism is necessarily
isotopic to the identity, and therefore each $K_\t$ is
diffeomorphic to the double of the Milnor fibre of $f$.

\begin{remark}\label{example-Jawad}
The hypothesis in Corollary \ref{Cor:Mil.fib} of $X$ being
irreducible avoids situations as in the following example.  Let
$g: \C^2 \to \C$ be defined by $g(x,y) = x \cdot y$. Its zero
locus $X $ consists of the two axis. Now   let $f: X \to \C$ be
given by $f(x,y) = x$. Then $V \subset X$ is the $y$-axe and each
$X_\t$ is
$$ X_\t = \{ (x,y) \in \C^2 \,| \,x = 0 \; \hbox{or} \; y = 0 \;  \hbox{and} \;  x = t e^{i
\t} \,, t \in \R \,\}.$$ One has $\,E_\t \cup E_{\t + \pi} = \{
(x,y) \in \C^2 \,| \,  y = 0 \;  \hbox{and} \; x = t e^{i \t} \,,
t \in \R \,\}.\,$ Thus no $z \in V \setminus \{\0\}$ is in the
closure of $E_\t \cup E_{\t + \pi}$.

\end{remark}

\begin{remark}
If $X$ is non-singular and $f$ has an isolated critical point with
respect to some Whitney stratification, then, as observed in
\cite{RS}, the family  $\{X_\t \}$ is (c)-regular in  Bekka's
sense (see \cite{Be}). In this case, one can follow the method of
\cite{RS} to construct two flows on the product $X_\t\setminus
\{0\} \times [0,\pi)$ which give on the one hand the
transversality of the $X_\t$'s with the spheres, and on the other
hand a flow interchanging the $X_\t\setminus \{0\}$'s, as in
Proposition~\ref{lem:vec.fld.sph}. This is similar to what we do
here on the blow-up, and in fact these considerations somehow
inspired our constructions.
\end{remark}

\section{The real analytic case}\label{section-Real-case}

We now look at real analytic mappings from the viewpoint of the
previous sections. Let $U$ be an open neighbourhood  of the origin
$\0$ in $\R^{n+2}$ and let $f: (U,\0) \to (\R^{2},0)$ be a locally
surjective, real analytic map. Set as before $V:= f^{-1}(0)$ and
denote by $K_\e= K$ the intersection $ V \cap \s_\e$. This is the
link of $V$, which is independent of $\e$ up to homeomorphism. We
assume further that $0 \in \R^{2}$ is the only critical value of
$f$, so the Jacobian matrix $Df(x)$ has rank $2$ for all $x \in U
\setminus V$.

\subsection{The strong Milnor condition}

We know from \cite{Milnor:ISH, Mi2} that if $f$  has an isolated
critical point at $\0$, then one has a fibration of the
Milnor-L\^e   type (\ref{ML-fib}), and this can always be taken
into a fibration  of the complement of $K$,   $\s_\e \setminus K \buildrel{\phi} \over
\longrightarrow \s^1$. So $K$ is a {\it
fibred knot}. But we also know from \cite{Mi2} that  the
projection map $\phi$ can not be always taken to be the obvious
map $f/\vert {f} \vert $. As noticed in \cite{Pichon-Seade}, these
remarks extend to the case when the real analytic map-germ $f$ is
assumed to have only an isolated critical value at $0 \in \R^{2}$
provided $V$ has dimension greater than $0$ and $f$ has the Thom property, {\it i.e.}, when there exists
a Whitney stratification of $U$ adapted to $V$, for which $f$
satisfies Thom's $a_f$-condition.

The following definitions extend those given in \cite{RSV}
when $f$ has an isolated critical point.

\begin{definition} Let $f: (U,\0) \to (\R^{2},0)$ be a locally surjective real analytic
map-germ.
\begin{enumerate}[i)]
\item We say that $f$ has the Milnor-L\^e property at $\0 \in U
\subset \R^{n+2}$ if it has an isolated critical value at $0 \in
\R^2$, $V$ has dimension more than $0$ and $f$ has the Thom property.

\item We say that $f$ has the strong Milnor property if for every
sufficiently small $\e >0$ one has a $C^{\infty}$ fibre bundle
$\s_\e \setminus K_\e \buildrel{\phi} \over \longrightarrow \s^1$,
where the projection map $\phi$ is $f/\vert {f} \vert $. (If one
considers map-germs defined on analytic varieties with singular
set of dimension greater than $0$, then this fibre bundle is
required to be only continuous.)
\end{enumerate}

\end{definition}

\subsection{$d$-Regularity for real analytic map-germs }

 Following the  construction above of a canonical pencil for holomorphic maps,
 for each line ${\cal L}_\t$ through $0 \in
\R^{2}$, let $X_\t= f^{-1}({\cal L}_\t)$. One has:

\begin{proposition} Each $X_\t$ is a real analytic
hypersurface  of $U$, of codimension $1$, such that:
\begin{itemize}

\item   Their union is $U$ and the intersection of any two
distinct $X_\t$'s is $V$.

\item   Each $X_\t$ is non-singular away from the singular
set of $V$, ${\rm Sing}(V)$.
\end{itemize}
\end{proposition}

The proof is an exercise and we leave it to the reader.

\begin{definition} The family $\{X_\t \vert {\cal L} \in \RP^{1}
\}$ is {\it the canonical pencil} of $f$.
\end{definition}

\begin{definition}\label{def.d-regularity} The map $f$ is {\it $d$-regular at $\0$} if there exist a positive definite metric
 $d: \R^{n+2} \to \R$, defined by some quadratic form,   and
 $\e > 0$ such
that every sphere (for the metric $d$) of radius $\le \e$ and
centred at $\0$ meets transversally each  $X_\t$. If we want to
emphasise the metric, then we say that $f$ is {\it $d$-regular at
$\0$ with respect to the given metric}.
\end{definition}

\begin{examples}
\begin{enumerate}[i)]

\item By \cite {Mi2} (see Lemma \ref{it:St.trans.Sp} above), every
holomorphic germ  $f:(\C^n,\0) \to (\C,0)$ is d-regular at $\0$
for the usual metric.

\item By \cite {Pichon-Seade}, given holomorphic germs $f,
g:(\C^2,\0) \to (\C,0)$ such that $f \bar g$ has an isolated
critical value at $0 \in \C$, the map $f \bar g$ is d-regular at
$\0$ for the usual metric. The same statement holds for the map
$f/g :(\C^n,\0) \to (\C,0)$ if we further demand that the
meromorphic germ $f/g$ be semi-tame (see \cite {Pichon-Seade} for
the definition and details).

\item By \cite{Se1}, every twisted Pham-Brieskorn polynomial
$z_1^{a_1} \, \bar z_{\sigma(1)} + \cdots   + z_n^{a_n} \, \bar
z_{\sigma(n)}$, where $\sigma$ is a permutation of $\{1,...,n\}$,
is d-regular at $\0$ for the usual metric. The same statement
holds for all quasi-homogeneous singularities, since the
$\R^+$-orbits are tangent to the $X_\t$. This applies, for
instance, to the singularities with polar action of \cite{Cisneros, Oka}.

\item By \cite{RS}, every map-germ $g\colon(\R^{n+2},\0) \to (\R^2,0)$ for
which its pencil is $c$-regular (in the sense of K. Bekka) with
respect to the control function defined by the metric $d$, is
$d$-regular. That was indeed one of the motivations to pursue this
research in the real analytic setting.

\end{enumerate}
\end{examples}

\subsection{The fibration theorem}

Let us prove Theorem \ref{theorem.real}, stated in the
introduction. Notice that statement \ref{it:Milnor.Le}) is
well-known (see for instance \cite{Pichon-Seade}), so we only need
to prove statements \ref{it:Strong.Milnor}) and
\ref{it:equivalence}).

Let us equip $U$ with a Whitney stratification adapted to $V$, and  consider the
restriction of $f$ to $ \B_\e \setminus V$, which is a submersion by hypothesis.

Define the real analytic map
\begin{equation*}
 \mathfrak{F}\colon \B_\e \setminus V\to \R^2\setminus\{0\}
\end{equation*}
by
\begin{equation*}
  \mathfrak{F}(x)= \norm{x} \frac{f(x)}{\norm {f(x)}}.
\end{equation*}
Notice that given $y =   \mathfrak{F}(x)$ in a line ${\cal L}_\t$
through $0$, the fibre $  \mathfrak{F}^{-1}(y)$ is the
intersection of the corresponding element $X_\t$ in the pencil
with the sphere of radius $\norm{x}$ centred at $\0$. Thus we call
$\mathfrak{F}$ the {\it spherefication} of $f$, as in the
holomorphic case.

\begin{lemma} If $f$ is $d$-regular, then $\mathfrak{F}$ is a submersion for all $x \in B_\e \setminus V$ with $\e>0$
sufficiently small.
\end{lemma}

The proof of this lemma is exactly the same as in Lemma~\ref{lem:g.sub}

\medskip

As in the holomorphic case, this lemma implies the   proposition
below; we leave the details to the reader. The only point to
notice is that because we are assuming the ambient space $U$ is
smooth, the liftings we use of vector fields from $\R^2 \setminus
{0}$ to $U$ can be taken to be $C^\infty$ and we do not need to
use Verdier's theory of rugose vector fields.

 \begin{proposition}\label{lemma} If   $f $ is $d$-regular for some metric $d$, then
 there exists  $\e > 0$ sufficiently small,
such that there exists  $C^\infty$ vector field
on $ \B_{\e} \setminus V$  such that :

\begin{enumerate}[i)]

\item Each of its integral lines    is contained in an element
$X_\t$ of the pencil;

 \item It is transverse to all $d$-spheres
around $\0$; and

\item It is transverse to all Milnor tubes $f^{-1}(\partial
\D_\delta)$, for all sufficiently small discs $\D_\delta$ centred
at $0 \in \R^{2}$.

\end{enumerate}
\end{proposition}

\medskip

\begin{remark}[Uniform Conical Structure] Notice that $d$-regularity for $f$ implies that
its canonical pencil has a uniform conical structure away from
$V$. In the holomorphic case one has uniform conical structure
everywhere near $\0$, this is part of the content of Theorem
\ref{Thm:Can.Dec}. In the real analytic case envisaged here, one
can prove this uniform conical structure everywhere near $\0$ if
we further demand that $f$ has the strict Thom property, which is
automatic for holomorphic maps into $\C$.
\end{remark}

\medskip

Let us now prove Theorem \ref{theorem.real}. Recall that we are
assuming $f$ has the Thom $a_f$ property. Thus for $ \e
>> \delta
> 0$ sufficiently small one has a {\it solid} Milnor tube,
$$SN(\e, \delta) := \B_\e \cap f^{-1}(\D_\delta \setminus \{0\})\,,$$
and a  fibre bundle:
$$ f : SN(\e, \delta) \longrightarrow \D_\delta \setminus \{0\} \,,$$
where $\B_\e$ is the ball in $\R^{n+2}$ of radius $\e$ and centre
$\0$ and $\D_\delta$ is the disc in $\R^2$ of radius $\delta$ and
centre $0$. The restriction of this locally trivial fibration to
the boundary of $\D_\delta$ gives the fibration in
 the first statement in Theorem
\ref{theorem.real}.

Now let $\pi_1: \D_\delta \setminus \{0\} \to \mathbb S^1$ be
defined by $t \mapsto t/ |t|$. Let $\pi_2:\mathbb S^1 \to \mathbb
R P^1$ be the canonical projection, and set
$$\Psi:=\pi_2 \circ \pi_1 \circ f : SN(\e, \delta)
\longrightarrow \R P^1\,.$$ This is a fibre bundle with fibres
$X_\t \cap SN(\e, \delta)$, and it yields to statement (ii) in
Theorem \ref{theorem.real} restricted to the solid Milnor tube
$SN(\e, \delta)$. We now use the vector field in Proposition
\ref{lemma} to complete the proof of Theorem \ref{theorem.real}.

\section*{APPENDIX: Proof of Lemma \ref{prop:arg.lambda}}

 Here we use the notations and hypotheses of the lemma in question.

\medskip
\noindent
{\bf PROOF OF LEMMA \ref{prop:arg.lambda}:}
 Suppose there were points $z\in S_\a$ arbitrarily close to the origin with
\begin{equation*}
 \pi_{\alpha_z}\bigl(\grad\log f(z)\bigr)=\lambda\pi_{\alpha_z}(z)\neq0,
\end{equation*}
    and with $\abs{\arg\lambda}$ strictly greater than $\pi/4$. In other words, $\lambda$
lies in the open half-plane
\begin{equation*}
 \Re\bigl((1+i)\lambda\bigr)<0,
\end{equation*}
or the open half-plane
\begin{equation*}
 \Re\bigl((1-i)\lambda\bigr)<0.
\end{equation*}
We want to express these conditions by \emph{real analytic} equalities and inequalities, so as to apply the
\emph{analytic curve selection lemma} \cite{BV}.

Let
\begin{equation*}
 W=\set{z\in S_\a}{\pi_{\alpha_z}\bigl(\grad\log f(z)\bigr)=\mu\pi_{\alpha_z}(z),\,\mu\in\C}.
\end{equation*}
Thus $z\in W$ if and only if the equations
\begin{equation*}
 \pi_{\alpha_z}(z)_j\Bigr(\pi_{\alpha_z}\bigl(\grad f(z)\bigr)_k\Bigr)
=\pi_{\alpha_z}(z)_k\Bigl(\pi_{\alpha_z}\bigl(\grad f(z)\bigr)_j\Bigr),
\end{equation*}
are satisfied, where the subindices $j$ and $k$ denote the $j$-th and $k$-th components.

Setting $z_j=x_j+y_j$ and taking real and imaginary parts, we
obtain a collection of real \emph{analytic} equations in the real
variables $x_j$ and $y_j$. This proves that $W\subset
S_\a\subset\C^n$ is a \emph{real analytic set}. Note that a point
$z\in S_\a$ belongs to $W$ if and only if
\begin{equation*}
 \pi_{\alpha_z}\bigl(\grad f(z)\bigr)/\bar{f}(z)=\lambda\pi_{\alpha_z}(z)
\end{equation*}
for some complex number $\lambda$. Multiplying by $\bar{f}(z)$ and
taking the inner product with $\bar{f}(z)\pi_{\alpha_z}(z)$, we
have
\begin{align*}
 \Bigl\langle\pi_{\alpha_z}\bigl(\grad f(z)\bigr),\bar{f}(z)\pi_{\alpha_z}(z)\Bigr\rangle
&=\Bigl\langle\lambda\bar{f}(z)\pi_{\alpha_z}(z),\bar{f}(z)\pi_{\alpha_z}(z)\Bigr\rangle\\
&=\lambda\norm{\bar{f}(z)\pi_{\alpha_z}(z)}^2.
\end{align*}
In other words, the number $\lambda$, multiplied by a positive real number, is equal to
\begin{equation*}
 \lambda'(z)=\Bigl\langle\pi_{\alpha_z}\bigl(\grad f(z)\bigr),\bar{f}(z)\pi_{\alpha_z}(z)\Bigr\rangle.
\end{equation*}
Hence
\begin{equation*}
 \arg\lambda=\arg\lambda'.
\end{equation*}
Clearly $\lambda'$ is a (complex valued) real \emph{analytic} function of the real variables $x_j$ and $y_j$.
Now let $U_+$ (respectively $U_-$) denote the open set consisting of all $z$ satisfying the real \emph{analytic} inequality
\begin{equation*}
 \Re\bigl((1+i)\lambda'(z) \bigr)<0
\end{equation*}
(respectively
\begin{equation*}
 \Re\bigl((1-i)\lambda'(z) \bigr)<0
\end{equation*}
for $U_-$).

We have assumed that there exist points $z$ arbitrarily close to the origin with
$z\in W\cap(U_+\cup U_-)$. Here by the \emph{analytic curve selection lemma} (see for instance \cite[Proposition 2.2]{BV}),
there must exist a real analytic path
\begin{equation*}
 p\colon[0,\e)\to S_\a\subset\C^n
\end{equation*}
with $p(0)=\0$ and with either
\begin{equation*}
 p(t)\in W\cap U_+
\end{equation*}
for all $t>0$, or
\begin{equation*}
 p(t)\in W\cap U_-
\end{equation*}
for all $t>0$. In either case, for each $t>0$ we get
\begin{equation*}
 \pi_{\alpha_t}\Bigl(\grad\log f\bigl(p(t)\bigr)\Bigr)=\lambda\pi_{\alpha_t}\bigl(p(t)\bigr)
\end{equation*}
with
\begin{equation*}
 \abs{\arg \lambda(t)}>\pi/4
\end{equation*}
which contradicts Lemma~\ref{lem:curve.strat}.
\qed




\begin{thebibliography}{00}


\bibitem{Be}
Karim~Bekka.
\newblock Regular quasi-homogeneous stratifications.
\newblock In {D. Trotman} and {L. C. Wilson}, editors, {\em Stratifications,
  singularities and differential equations, II (Marseille, 1990; Honolulu, HI,
  1990)}, volume~55 of {\em Travaux en Cours}, pages 1--14. Hermann, Paris,
  1997.

\bibitem{BLS}
Jean-Paul Brasselet, L\^e D\~ung Tr\'ang, and Jos\'e Seade.
\newblock Euler obstruction and indices of vector fields.
\newblock {\em Topology}, 39:1193-1208, 2000.

\bibitem{Brianson-Maisonobe-Merle:LSDSWCT}
Jo{\"e}l Brian{\c{c}}on, Philippe Maisonobe, and Michel Merle.
\newblock Localisation de syst\`emes diff\'erentiels, stratifications de
  {W}hitney et condition de {T}hom.
\newblock {\em Invent. Math.}, 117(3):531--550, 1994.

\bibitem{BV}
Dan Burghelea and Andrei Verona.
\newblock Local homological properties of analytic sets.
\newblock {\em Manuscripta Math.}, 7:55--66, 1972.

\bibitem{Cisneros} Jos{\'e} Luis Cisneros-Molina,
\textit{Join theorem for polar weighted homogeneous
singularities}.
\newblock  In ``Singularities II. Geometric and topological aspects". Proceedings
of the   international conference in honor of the 60th Birthday of
L\^e D\~ung Tr\'ang, Cuernavaca, Mexico, January
 2007.
\newblock  Ed. J.-P. Brasselet   et al. AMS
Contemporary Mathematics 475, 43-59 (2008).

\bibitem{CSS:MR} J. L. Cisneros-Molina, J. Snoussi, J. Seade,
Milnor Fibrations and $d$-regularity for real analytic Singularities,
\newblock Preprint, 2008, to appear in International Journal of
Mathematics.

\bibitem{Denkowska-Wachta:CSSACw}
Zofia Denkowska and Krystyna Wachta.
\newblock Une construction de la stratification sous-analytique avec la
  condition (w).
\newblock {\em Bull. Polish Acad. Sci. Math.}, 35(7-8):401--405, 1987.

\bibitem{Durfee:NAS}
Alan~H. Durfee.
\newblock Neighborhoods of algebraic sets.
\newblock {\em Trans. Amer. Math. Soc.}, 276(2):517--530, 1983.

\bibitem{GWPL}
C.~G. Gibson, K.~Wirthm{\"u}ller, A.~A. du~Plessis, E.~N. Looijenga,
  Topological stability of Smooth Mappings, Lecture Notes in Mathematics 552,
  Springer Verlag, 1976.

\bibitem{Goresky-MacPherson:SMT}
Mark Goresky and Robert MacPherson.
\newblock {\em Stratified {M}orse theory}, volume~14 of {\em Ergebnisse der
  Mathematik und ihrer Grenzgebiete (3) [Results in Mathematics and Related
  Areas (3)]}.
\newblock Springer-Verlag, Berlin, 1988.

\bibitem{Hamm}
Helmut Hamm.
\newblock Lokale topologische Eigenschaften komplexer R\"aume,
\newblock {\em Math. Ann.}, 191:235--252, 1971.

\bibitem{Hironaka:SubAnal}
Heisuke Hironaka.
\newblock Subanalytic sets.
\newblock In {\em Number theory, algebraic geometry and commutative algebra, in
  honor of Yasuo Akizuki}, pages 453--493. Kinokuniya, Tokyo, 1973.

\bibitem{Hironaka:SF}
Heisuke Hironaka.
\newblock Stratification and flatness.
\newblock In {\em Real and complex singularities (Proc. Ninth Nordic Summer
  School/NAVF Sympos. Math., Oslo, 1976)}, pages 199--265. Sijthoff and
  Noordhoff, Alphen aan den Rijn, 1977.

\bibitem{Kuo:RTAWS}
Tzee-Char Kuo.
\newblock The ratio test for analytic {W}hitney stratifications.
\newblock In {\em Proceedings of Liverpool Singularities-Symposium, I
  (1969/70)}, Lecture Notes in Mathematics, Vol. 192, pages 141--149, Berlin,
  1971. Springer.

\bibitem{Le1}
 D\~ung~Tr\'ang L\^e.
\newblock Some remarks on relative monodromy.
\newblock In {P. Holm}, editor, {\em Real and complex singularities (Proc.
  Ninth Nordic Summer School/NAVF Sympos. Math., Oslo, 1976)}, pages 397--403.
  Sijthoff and Noordhoff, Alphen aan den Rijn, 1977.

\bibitem{Le:VCCAS}
D\~ung~Tr\'ang L\^e.
\newblock Vanishing cycles on complex analytic sets.
\newblock In ''Various problems in algebraic analysis''. Proc. Sympos., Res. Inst.
  Math. Sci., Kyoto Univ., Kyoto, 1975.
\newblock {\em S\^urikaisekikenky\^usho K\'oky\^uroku}, (266):299--318, 1976.


\bibitem{Loj}
Stanis{\l}aw {\L}ojasiewicz.
\newblock Ensambles semi-analytique.
\newblock Cours donn{\'e} {\`a} la {F}acult{\'e} des Sciences d'Orsay, polycopi{\'e} de l'I.H.E.S., Bures-sur-Yvette, France, 1965.


\bibitem{Lojasiewicz-Stasica-Wachta:SSACV}
Stanis{\l}aw {\L}ojasiewicz, Jacek Stasica, and Krystyna Wachta.
\newblock Stratifications sous-analytiques. {C}ondition de {V}erdier.
\newblock {\em Bull. Polish Acad. Sci. Math.}, 34(9-10):531--539 (1987), 1986.

\bibitem{Massey} David B. Massey
\newblock  A Strong {\L}ojasiewicz Inequality and Real Analytic Milnor Fibrations
\newblock {\em  arXiv:math/0703613}, 2008.

\bibitem{Mather:TopStab}
John Mather.
\newblock Notes on {T}opological {S}tability.
\newblock Harvard University, July 1970.

\bibitem{Milnor:ISH}
John Milnor.
\newblock On isolated singularities of hypersurfaces.
\newblock Preprint June 1966. Unpublished.

\bibitem{Mi2}
John Milnor.
\newblock {\em Singular points of complex hypersurfaces}.
\newblock Annals of Mathematics Studies, No. 61. Princeton University Press,
  Princeton, N.J., 1968.

\bibitem{Oka} Mutsuo Oka
\newblock Topology of polar weighted homogeneous hypersurfaces
\newblock  Kodai Math. J. 31, No. 2, 163-182 (2008).

\bibitem{Par}
Adam Parusi{\'n}ski.
\newblock Limits of tangent spaces to fibres and the {$w\sb f$} condition.
\newblock {\em Duke Math. J.}, 72(1):99--108, 1993.

\bibitem{Pichon-Seade} Anne Pichon, Jos{\'e} Seade.
\newblock Fibred Multilinks and singularities $f \bar g$.
\newblock Math. Ann. 342,   487-514 (2008).

\bibitem{RS}
Maria Aparecida~Soares Ruas and Raimundo Nonato~Ara{\'u}jo dos Santos.
\newblock Real {M}ilnor fibrations and (c)-regularity.
\newblock {\em Manuscripta Math.}, 117(2):207--218, 2005.

\bibitem{RSV}
Maria Aparecida~Soares Ruas, Jos{\'e} Seade, and Alberto Verjovsky.
\newblock On real singularities with a {M}ilnor fibration.
\newblock In A.~Libgober and M.~Tibar, editors, {\em Trends in singularities},
  Trends Math., pages 191--213. Birkh\"auser, Basel, 2002.

\bibitem{Raimundo1} Raimundo Ara\'ujo dos Santos
\newblock  Uniform (m)-condition and Strong Milnor fibrations
\newblock  In ``Singularities II. Geometric and topological aspects". Proceedings
of the   international conference in honor of the 60th Birthday of
L\^e D\~ung Tr\'ang, Cuernavaca, Mexico, January
 2007.
\newblock  Ed. J.-P. Brasselet   et al. AMS
Contemporary Mathematics 475, 189-198  (2008).


\bibitem{Raimundo-Tibar} Raimundo Ara\'ujo dos Santos, Mihai Tib\v ar
\newblock Real map germs and higher open books
\newblock {\em arXiv:0801.3328}, 2008.


\bibitem{Sch}
Marie-H{\'e}l{\`e}ne Schwartz.
\newblock {\em Champs radiaux sur une stratification analytique}, volume~39 of
  {\em Travaux en Cours [Works in Progress]}.
\newblock Hermann, Paris, 1991.

\bibitem{Se1}
Jos{\'e} Seade.
\newblock Open book decompositions associated to holomorphic vector fields.
\newblock {\em Bol. Soc. Mat. Mexicana (3)}, 3(2):323--335, 1997.

\bibitem{Se2} Jos{\'e} Seade.
\newblock  On Milnor�s Fibration Theorem for Real and Complex Singularities
\newblock {\em In ``Singularities in Geometry and Topology"}, Proceedings of the Trieste
Singularity Summer School and Workshop,
\newblock {\em World Scientific}, 2007, p. 127-158.

\bibitem{Se3}
Jos{\'e} Seade.
\newblock {\em On the topology of isolated singularities in analytic spaces},
  volume 241 of {\em Progress in Mathematics}.
\newblock Birkh\"auser Verlag, Basel, 2006.

\bibitem{Teissier:VarPolII}
Bernard Teissier.
\newblock Vari\'et\'es polaires. {II}. {M}ultiplicit\'es polaires, sections
  planes, et conditions de {W}hitney.
\newblock In {\em Algebraic geometry (La R\'abida, 1981)}, volume 961 of {\em
  Lecture Notes in Math.}, pages 314--491. Springer, Berlin, 1982.

\bibitem{Thom:EMS}
Ren\'e~Thom.
\newblock Ensembles et morphismes stratifi\'es.
\newblock {\em Bull. Amer. Math. Soc.}, 75:240--284, 1969.

\bibitem{Ver}
Jean-Louis Verdier.
\newblock Stratifications de {W}hitney et th\'eor\`eme de {B}ertini-{S}ard.
\newblock {\em Invent. Math.}, 36:295--312, 1976.

\bibitem{Whitney:TAV}
Hassler Whitney.
\newblock Tangents to an analytic variety.
\newblock {\em Ann. of Math. (2)}, 81:496--549, 1965.
\end{thebibliography}
\end{document}